\documentclass[10pt,a4paper,oneside]{amsart}

\usepackage{latexsym}

\usepackage{amsfonts}

\usepackage{hyperref}

\usepackage{euscript}

\usepackage{amscd}

\usepackage{graphicx}

\swapnumbers
\newtheorem{theorem}{Theorem}[section]
\newtheorem{proposition}[theorem]{Proposition}
\newtheorem{lemma}[theorem]{Lemma}
\newtheorem{corollary}[theorem]{Corollary}
\newtheorem{conjecture}[theorem]{Conjecture}
\newtheorem*{theorem*}{Theorem}
\newtheorem*{lemma*}{Lemma}
\newtheorem*{claim*}{Claim}
\newtheorem*{ConjLemma*}{Conjugation Lemma}

\theoremstyle{remark}
\newtheorem{remark}[theorem]{Remark}
\newtheorem{remarks}[theorem]{Remarks}
\newtheorem{example}[theorem]{Example}
\newtheorem{examples}[theorem]{Examples}

\newcommand{\MathRoman}[1]{\mathop{\mathrm{#1}}\nolimits}
\newcommand{\mathcan}[1]{\mathbb{#1}}

\newcommand{\Aff}{\mathcan{A}}
\newcommand{\CC}{\mathcan{C}}
\newcommand{\FF}{\mathcan{F}}

\newcommand{\NN}{\mathcan{N}}
\newcommand{\LL}{\mathcan{L}}
\newcommand{\PP}{\mathcan{P}}
\newcommand{\QQ}{\mathcan{Q}}

\newcommand{\ZZ}{\mathcan{Z}}

\DeclareMathOperator{\Alb}{Alb}
\DeclareMathOperator{\codim}{codim}
\DeclareMathOperator{\haut}{ht}

\DeclareMathOperator{\Gal}{Gal}
\DeclareMathOperator{\GL}{\mathrm{GL}}
\DeclareMathOperator{\Hom}{Hom}
\DeclareMathOperator{\Jac}{Jac}

\DeclareMathOperator{\Pic}{Pic}
\DeclareMathOperator{\ProjFunc}{Proj}
\DeclareMathOperator{\reg}{Reg}
\DeclareMathOperator{\red}{red}
\DeclareMathOperator{\sing}{Sing}
\DeclareMathOperator{\spec}{Spec}
\DeclareMathOperator{\Tr}{Tr}

\newcommand{\AlbUniv}{\MathRoman{(\mathbf{Alb})}}
\newcommand{\abs}[1]{\Bigl\lvert\,#1\,\Bigr\rvert}
\newcommand{\set}[2]{\left\{#1\,\mid\,#2\right\}}
\newcommand{\card}[1]{\lvert#1\rvert}
\newcommand{\Fq}{\mathop{\FF_{q}}\nolimits}
\newcommand{\idest}{\textit{i.\,e.}}
\newcommand{\indep}{\mathbf{M}}
\newcommand{\isom}{\stackrel{\sim}{\longrightarrow}}
\newcommand{\loccit}{\textit{loc.~cit.}}
\newcommand{\PicS}{\underline{\MathRoman{Pic}}}
\newcommand{\ProjCan}{\PP^{N}}
\newcommand{\ProjDual}{\widehat{\PP}^{N}}
\newcommand{\qedbox}{\hfill $\ \Box$}
\newcommand{\resol}{\MathRoman{(\mathbf{R}_{n,p})}}
\newcommand{\resolTwo}{\MathRoman{(\mathbf{RS2})}}
\newcommand{\smallbinom}[2]{\textstyle{\binom{#1}{#2}}}
\newcommand{\TotDeg}{\MathRoman{tot.deg}}

\renewcommand{\mod}{\MathRoman{mod}}
\renewcommand{\Im}{\MathRoman{Im}}

\begin{document}

% Title
% -----

\title{\'Etale cohomology, Lefschetz Theorems \\
and Number of Points of Singular Varieties \\
over Finite Fields}

\renewcommand{\rightmark}
{\textsc\large{SINGULAR VARIETIES OVER FINITE FIELDS}}

\dedicatory{
Dedicated to Professor Yuri Manin for his 65th birthday$^{*}$: \\
\includegraphics[height = 13mm]{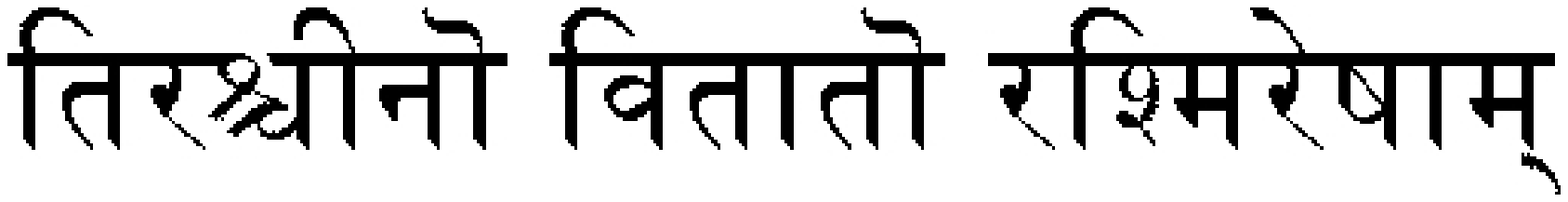}\\
}
\thanks{$^{*}$ `` tira\'sc\={\i}no vitato ra\'smir e\d{s}\={a}m "
(Their cord was extended across) (\d{R}g Veda X.129)}

\author{Sudhir R. Ghorpade}
\thanks{$\dagger$ The first author is partially supported by a `Career
Award' grant from AICTE, New Delhi and an IRCC grant from IIT Bombay.}
\address{S.R.G.: Department of Mathematics
\newline \indent
Indian Institute of Technology, Bombay
\newline \indent
Powai, Mumbai 400076 - INDIA}
\email{srg@math.iitb.ac.in}

\author{Gilles Lachaud}
\address{G.L.: Institut de Math\'ematiques de Luminy
\newline \indent
CNRS
\newline \indent
Luminy Case 907, 13288 Marseille Cedex 9 - FRANCE}
\email{lachaud@iml.univ-mrs.fr}


\date{August 6, 2008. This is a corrected, revised and updated version of 
the paper 
published in  \emph{Mosc. Math. J.}  {\bfseries{2}}  (2002), 589--631. \MR{1988974 (2004d:11049a})}

\begin{abstract}
We prove a general inequality for estimating the number of points of
arbitrary complete intersections over a finite field. This extends a result
of Deligne for nonsingular complete intersections. For normal complete
intersections, this inequality generalizes also the classical Lang-Weil
inequality. Moreover, we prove the Lang-Weil inequality for affine as well
as projective varieties with an explicit description and a bound for the
constant appearing therein. We also prove a conjecture of Lang and Weil
concerning the Picard varieties and \'etale cohomology spaces of projective
varieties. The general inequality for complete intersections may be viewed
as a more precise version of the estimates given by Hooley and Katz. The
proof is primarily based on a suitable generalization of the Weak Lefschetz
Theorem to singular varieties together with some Bertini-type arguments and
the Grothendieck-Lefschetz Trace Formula. We also describe some auxiliary
results concerning the \'etale cohomology spaces and Betti numbers of
projective varieties over finite fields and a conjecture along with some
partial results concerning the number of points of projective algebraic sets
over finite fields.
\end{abstract}


\maketitle

\newpage

~\bigskip \bigskip \bigskip \bigskip

\tableofcontents

\newpage

% body of the article
% -------------------

\section*{Introduction}
\label{Intro}

This paper has roughly a threefold aim. The first is to prove the following
inequality for estimating the number of points of complete intersections (in
particular, hypersurfaces) which may possibly be singular:
\begin{equation}
\label{Neweq:1}
\abs{\card{X(\Fq)} - \, \pi_{n}} \leq b'_{n - s - 1}(N - s - 1,\mathbf{d})
\, q^{(n + s + 1)/2} + C_{s}(X) q^{(n + s)/2}
\end{equation}
(cf. Theorem \ref{MainThm}). Here, $X$ denotes a complete intersection in
$\ProjCan$ defined over the finite field $k = \Fq$ of $q$ elements, $n$ the
dimension of $X$, $\mathbf{d} = (d_{1},\dots, d_{r})$ the multidegree of $X$
and $s$ an integer such that $\dim \sing X \leq s \leq n - 1$. Note that
$r = N - n$. Moreover, $\pi_{n}$ denotes the number of points of
$\PP^n(\Fq)$,
viz., $\pi_{n} = q^{n} + q^{n - 1} + \dots + 1$, and for any nonnegative
integers $j$ and $M$ with $M - j = r$, we denote by $b'_{j}(M, \mathbf{d})$
the primitive
$j$-th Betti number of a \emph{nonsingular} complete intersection in
$\PP^{M}$
of dimension $j$. This primitive Betti number is explicitly given by the
formula
\begin{equation}
\label{Neweq:2}
b'_{j}(M, \mathbf{d}) = (- 1)^{j + 1} (j + 1) + (- 1)^{N}
\sum_{c = r}^{N} (- 1)^{c} \binom{N + 1}{c + 1}
\sum_{\mathbf{\nu} \in M(c)} \ \mathbf{d}^{\mathbf{\nu}}
\end{equation}
where $M(c)$ denotes the set of $r$-tuples $\nu = (\nu_1 , \dots , \nu_r)$
of
positive integers such that $\nu_1 + \dots + \nu_r = c$ and
$\mathbf{d}^{\mathbf{\nu}} = {d_1}^{\nu_1}\cdots {d_r}^{\nu_r}$ for any such
$r$-tuple $\nu$. If we let $\delta = \max(d_{1}, \dots, d_{r})$ and $d =
d_{1} \cdots d_{r} = \deg X$, then it will be seen that
\begin{equation}
\label{Neweq:2bis}
b'_{j}(M, \mathbf{d}) \le (- 1)^{j + 1} (j + 1) + d \binom{M + 1}{j}
(\delta + 1)^{j} \leq \binom{M + 1}{j} (\delta + 1)^{M}.
\end{equation}
Lastly, $C_{s}(X)$ is a constant which is independent of $q$ (or of $k$). We
can take $C_s(X) = 0$ if $s= -1$, \idest, if $X$ is nonsingular, and in
general, we have
\begin{equation}
\label{Neweq:3}
C_{s}(X) \leq 9 \times 2^{r} \times (r \delta + 3)^{N + 1}.
\end{equation}
where $\delta$ is as above. The inequalities \eqref{Neweq:1},
\eqref{Neweq:2bis} and \eqref{Neweq:3} are proved in sections \ref{NbPtsCI},
\ref{CentralBetti} and \ref{TraceForm} respectively.

Our second aim is to discuss and elucidate a number of results related to
the \emph{conjectural statements} of Lang and Weil \cite{LangWeil}. These
concern the connections between the various \'etale cohomology spaces
(especially, the first and the penultimate) and the Picard (or the Albanese)
varieties of normal projective varieties defined over $\Fq$. For example, if
$X$ is any projective variety of dimension $n$, and $P^{+}_{2n - 1}(X,T)$ is
the characteristic polynomial of the piece of maximal filtration of $H^{2n -
1}(X, \QQ_{\ell})$ and $f_c(\Alb_{w} X, T)$ is the characteristic polynomial
of the Albanese-Weil variety $\Alb_{w} X$ of $X$, then we show that
\begin{equation}
\label{Neweq:4} P^{+}_{2n - 1}(X, T) = q^{-g}f_{c}(\Alb_{w} X, q^{n}T)
\end{equation}
where $g = \dim \Alb_{w} X$. In particular, the ``$(2n-1)$-th virtual Betti
number'' of $X$ is twice the dimension of $\Alb_{w} X$, and is independent
of $\ell$. These results are discussed in details in sections
\ref{Penultimate}, \ref{RelationAlbanese} and \ref{LangWeilConj}.

The third aim is to prove an effective version of the Lang-Weil inequality
\cite{LangWeil}. Recall that the Lang-Weil inequality states that if
$X \subset \ProjCan$ is any projective variety defined over $\Fq$ of
dimension $n$ and degree $d$, then with $\pi_n$ as above, we have
\begin{equation}
\label{Neweq:5}
\abs{\card{X(\Fq)} - \pi_{n}} \leq
(d - 1)(d - 2) q^{n - (1/2)} + C q^{(n -1)},
\end{equation}
where $C$ is a constant depending only on $N$, $n$ and $d$. The said
effective version consists in providing computable bounds for the constant
$C$ appearing in \eqref{Neweq:5}. For example, if $X$ is defined by the
vanishing of $m$ homogeneous polynomials in $N+1$ variables of degrees
$d_1, \dots , d_m$ and $\delta = \max(d_{1}, \dots, d_{m})$, then we show
that the constant $C$ in \eqref{Neweq:5} may be chosen such that
\begin{equation}
\label{Neweq:6}
C \leq 9 \times 2^{m} \times (m \delta + 3)^{N + 1}.
\end{equation}
We also prove an analogue of the Lang-Weil inequality for affine
varieties. These inequalities are proved in section \ref{LWIneq}. Lastly, in
section \ref{AlgSets}, we describe some old, hitherto unpublished, results
concerning certain general bounds for the number of points of projective
algebraic sets, as well as a conjecture related to the same.

We shall now describe briefly the background to these results and some
applications. For a more leisurely description of the background and
an expository account of the main results of this paper, we refer
to \cite{Chandigarh}.

In \cite{Weil48}, Weil proved a bound for the number of points of a
nonsingular curve over a finite field $\Fq$, namely that it differs from
$\pi_1 = q + 1$ by at most $2 g q^{1/2}$, where $g$ is the genus of the
curve. In \cite{Weil49}, he formulated the conjectures about the number of
points of varieties of arbitrary dimension. Before these conjectures became
theorems, Lang and Weil \cite{LangWeil} proved the inequality
\eqref{Neweq:5} in 1954.  In 1974, Deligne \cite{Deligne2} succeeded in
completing the proof of Weil Conjectures by establishing the so called
Riemann hypothesis for nonsingular varieties of any dimension. Using this
and
the Lefschetz Trace Formula, also conjectured by Weil, and proved by
Grothendieck, he obtained a sharp inequality for the number of
points of a nonsingular complete intersection. Later, in 1991, Hooley
\cite{Ho} and Katz \cite{Katz1} proved that if $X$ is a complete
intersection with a singular locus of dimension $s$, then
$$\card{X(\Fq)} - \pi_{n} = O(q^{(n + s + 1)/2}).$$
The inequality \eqref{Neweq:1} which we prove here may be regarded as a
more precise version of this estimate. In effect, we explicitly obtain the
coefficient of the first term and a computable bound for the coefficient of
the second term in the asymptotic expansion of the difference $\card{X(\Fq)}
-\pi_{n}$. When $s=-1$, \idest, when $X$ is nonsingular, then
\eqref{Neweq:1} is precisely the inequality proved by Deligne
\cite{Deligne2}. On the other hand, if $X$ is assumed normal, then we can
take $s=n-2$ and \eqref{Neweq:1} implies the Lang-Weil inequality for normal
complete intersections. The explicit formula \eqref{Neweq:2} is a
consequence of the work of Hirzebruch \cite{Hirzebruch} and Jouanolou
\cite{Jouanolou2} on nonsingular complete intersections, while the
computable bound \eqref{Neweq:3} is obtained using some work of Katz
\cite{Katz4} on the sums of Betti numbers. Corollaries of
\eqref{Neweq:1} include some results of Shparlinski\u{\i} and Skorobogatov
\cite{Sh-S} for complete intersections with at most isolated singularities
as well as some results of Aubry and Perret \cite{AubryPerret} for singular
curves. It may be remarked that the recent article \cite{Katz3} of Katz has
a more general purpose than that of Section \ref{NbPtsCI} of this article,
since it provides bounds for exponential sums defined over singular
varieties. However, as far as the number of points of singular varieties are
concerned, the bounds presented here are more precise than those obtained by
specializing the results of \cite{Katz3}.

After proving \eqref{Neweq:5}, Lang and Weil \cite{LangWeil} observed that
if $K$ is an algebraic function field of dimension $n$ over $k= \Fq$,  then
there is a constant $\gamma$ for which \eqref{Neweq:5} holds with  $(d -
1)(d - 2)$ replaced by $\gamma$, for any model $X$ of $K/k$, and moreover, 
the smallest such constant $\gamma$ is a birational invariant. They also 
noted that the zeros and poles of the zeta function $Z(X,T)$ in the open 
disc $\vert T \vert < q^{-(n-1)}$ are birational invariants, and that in  the
smaller disc $\vert T \vert < q^{-(n - 1/2)}$ there is exactly one pole  of
order $1$ at $T = q^{-n}$. Then they wrote: {\it about the behaviour of
$Z(X,T)$ for $\vert T \vert \ge q^{-(n-1/2)}$, we can only make the
following  conjectural statements, which complement the conjectures of
Weil}.  These statements are to the effect that when $X$ is complete and
nonsingular, the quotient
$$\frac{Z(X,T)( 1 - q^{n}T)}{f_{c}(P, T)}$$
has no zeros or poles inside $\vert T \vert < q^{-(n-1)}$ and at least one
pole on $\vert T \vert=q^{-(n-1)}$, where $P$ denotes the Picard variety of
$X$. Moreover, with $K/k$ and $\gamma$ as above, we have $\gamma = 2 \dim
P$. Lang and Weil \cite{LangWeil} proved that these statements are valid in
the case of complete nonsingular curves, using the Riemann hypothesis for
curves over finite fields. As remarked by  Bombieri and Sperber \cite[p.
333]{BS}, some of the results conjectured by Lang and Weil are apparently
known to the experts but one is unable to locate formal  proofs in the
literature.  Some confusion is also added by the fact that there are, in
fact, two notions of the Picard variety of a projective variety $X$. These
notions differ when $X$ is singular  and for one of them, the analogue of
\eqref{Neweq:4} is false. Moreover, the proof in \cite{BS} of a part of the
Proposition on p. 133 appears to be incomplete (in dimensions $\ge 3$). In
view of this, we describe in some details results such as \eqref{Neweq:4}
which together with the Grothendieck-Lefschetz trace formula, prove the
conjecture of Lang and Weil, as well as several related results.  These
include the abovementioned Proposition of Bombieri and Sperber \cite[p. 
333]{BS}, which is proved here using a different method. An equality such as
\eqref{Neweq:4} can also be of interest as a ``motivic'' result in the sense
indicated in an early letter of Grothendieck \cite{Groth0}.

For the Lang-Weil inequality \eqref{Neweq:5}, it is natural to try to give a
proof using the trace formula. Some auxiliary results are still needed but
many of these are obtained in the course of proving \eqref{Neweq:1} and
\eqref{Neweq:4}. This, then, leads to the `effective version' and an affine
analogue. The latter yields, for example, a version of a lower bound due to
Schmidt \cite{Schmidt} for the number of points on affine hypersurfaces. For
varieties of small codimension, we obtain an improved version of the
Lang-Weil inequality.

\section{Singular Loci and Regular Flags}
\label{SingLoc}

We first settle notations and terminology. We also state some preliminary
results, and the proofs are omitted. Let $k$ be a field of any
characteristic $p \geq 0$ and $\overline{k}$ the algebraic closure of $k$. We
denote by $S = k[X_{0}, \dots, X_{N}]$ the graded algebra of polynomials in
$N + 1$ variables and by $\ProjCan = \ProjCan_{k} = \ProjFunc S$ the
projective space of dimension $N$ over $k$. By an \emph{algebraic variety
over $k$} we mean a separated scheme of finite type over $k$ which is
geometrically irreducible and reduced, \idest, geometrically integral, and by
a \emph{curve} we mean an algebraic variety of dimension one. In this paper,
we use the word \emph{scheme} to mean a scheme of finite type over $k$.

Recall that a point $x$ in a scheme $X$ is \emph{regular} if the local ring
$\mathcal{O}_{x}(X)$ is a regular local ring and \emph{singular} otherwise.
The \emph{singular locus} $\sing X$ of $X$ is the set of singular points of
$X$ ; this is is a closed subset of $X$  \cite[Cor. 6.12.5, p. 166]{EGA42}.
We denote by $\reg X$ the complementary subset of $\sing X$ in
$X$. Let $m \in \NN$ with $m \leq \dim X$. One says that $X$ is
\emph{regular in codimension} $m$ if it satisfies the following equivalent
conditions:
\begin{enumerate}
\item
\label{Rm}
every point $x \in X$ with $\dim \mathcal{O}_{x}(X) \leq m$ is regular.
\item
$\dim X - \dim \sing X \geq m + 1.$
\end{enumerate}
Condition \eqref{Rm} is called \emph{condition} $(R_{m})$ \cite[D\'ef.
5.8.2, p. 107]{EGA42}. A scheme is reduced if and only if it has no embedded
components and satisfies condition $(R_{0})$ \cite[5.8.5, p. 108]{EGA42}. A
scheme $X$ of dimension $n$ is \emph{regular} if it satisfies condition
$(R_{n})$, hence $X$ is regular if and only if $\dim \sing X = - 1$ (with
the convention $\dim \emptyset = - 1$).

For any scheme $X$, we denote by $\bar{X} = X \otimes_{k} \overline{k}$
the scheme deduced from $X$ by base field extension from $k$ to
$\overline{k}$. A variety $X$ is \emph{nonsingular} if $\bar{X}$ is regular.
If $k$ is perfect, then the canonical projection from $\bar{X}$ to $X$ sends
$\sing \bar{X}$ onto $\sing X$ \cite[Prop. 6.7.7, p. 148]{EGA42}. Hence,
$\bar{X}$ is regular in codimension $m$ if and only if $X$ has the same
property, and $\dim \sing \bar{X} = \dim \sing X$.

Let $R$ be a regular noetherian local ring, $A = R/\mathfrak{I}$ a quotient
subring of $R$, and $r$ the minimum number of generators of $\mathfrak{I}$.
Recall that $A$ is a \emph{complete intersection in $R$} if $\haut
\mathfrak{I} = r$. A closed subscheme $X$ of a regular noetherian scheme $V$
over a field $k$ is a \emph{local complete intersection at a point} $x \in
X$ if the local ring $\mathcal{O}_{x}(X)$ is a complete intersection in
$\mathcal{O}_{x}(V)$. The subscheme $X$ of $V$ is a \emph{local complete
intersection} if it is a local complete intersection at every closed point;
in this case it is a local complete intersection at every point, since the
set of $x \in X$ such that $X$ is a local complete intersection at $x$ is
open. A regular subscheme of $V$ is a local complete intersection. A
connected local complete intersection is Cohen-Macaulay \cite[Prop. 8.23,
p.186]{Ha}, hence equidimensional since all its connected components have
the same codimension $r$.

Let $X$ and $Y$ be a pair of closed subschemes of $\ProjCan$ with $\dim X
\geq \codim Y$, and $x$ a closed point of $X \cap Y$. We say that $X$ and
$Y$ \emph{meet transversally} at $x$ if
\begin{equation}
\label{Trans}
x \in \reg X \cap \reg Y \quad \text{and} \quad
\dim T_{x}(X) \cap T_{x}(Y) = \dim T_{x}(X) - \codim T_{x}(Y),
\end{equation}
and that they \emph{intersect properly} at $x$ if
\begin{equation}
\label{Proper}
\dim \mathcal{O}_{x}(X \cap Y) = \dim \mathcal{O}_{x}(X) - \codim
\mathcal{O}_{x}(Y).
\end{equation}
If $X$ and $Y$ are equidimensional, then they intersect properly at every
point of an irreducible component $Z$ if and only if
$$\dim Z = \dim X - \codim Y.$$
If this is fulfilled, one says that $Z$ is a \emph{proper component} of
$X \cap Y$ or that $X$ and $Y$ \emph{intersect properly} at $Z$. If every
irreducible component of $X \cap Y$ is a proper component, one says in this
case that $X$ and $Y$ \emph{intersect properly}. For instance, if $X$ is
irreducible of dimension $\geq 1$ and if $Y$ is a hypersurface in $\ProjCan$
with $X_{\red} \not\subset Y_{\red}$, then $X$ and $Y$ intersect properly by
Krull's Principal Ideal Theorem.

If $Y$ is a local complete intersection, then $X$ and $Y$ meet
trans\-versally at $x$ if and only if $x \in \reg (X \cap Y)$ and they
intersect properly at $x$.

\begin{lemma}
\label{IneqSing}
Assume that $X$ is equidimensional, $Y$ is a local complete intersection,
and $X$ and $Y$ intersect properly.
\begin{enumerate}
\item
\label{IneqSing1}
If $N(X,Y)$ is the set of closed points of $\reg X \cap \reg Y$ where $X$
and $Y$ do not meet transversally, then
$$
\sing (X \cap Y) = (Y \cap \sing X) \cup (X \cap \sing Y) \cup N(X,Y).
$$
\item
\label{IneqSing3}
If $X \cap Y$ satisfies condition $(R_{m})$, then $X$ and $Y$ satisfy it as
well.
\hfill \qedbox
\end{enumerate}
\end{lemma}

Let $X$ be a subscheme of $\ProjCan$. We say a subscheme $Z$ of $X$ is a
\emph{proper linear section} of $X$ if there is a linear subvariety $E$ of
$\ProjCan$ properly intersecting $X$ such that $X \cap E = Z$, in such a way
that $\dim Z = \dim X - \codim E$. If $\codim E = 1$, we say $Z$ is a
\emph{proper hyperplane section} of $X$. One sees immediately from Lemma
\ref{IneqSing} that if $X$ is an equidimensional subscheme of $\ProjCan$,
and \emph{if there is a nonsingular proper linear section of dimension $m$
of $X$ $(0 \leq m \leq \dim X)$, then condition $(R_{m})$ holds for $X$}. In
fact, these two conditions are equivalent, as stated in Corollary
\ref{Equiv}.

\emph{From  now on we assume that $k = \overline{k}$ is algebraically
closed}. We state here a version of Bertini's Theorem and some of its
consequences ; an early source for this kind of results is, for instance,
\cite[Sec. I ]{Weil54}. Let $\ProjDual$ be the variety of hyperplanes of
$\ProjCan$. Let $X$ be a closed subvariety in $\ProjCan$ and
$\EuScript{U}_{1}(X)$ be the set of $H \in \ProjDual$ satisfying the
following conditions:
\begin{enumerate}
\item
\label{Bt1}
$X \cap H$ is a proper hyperplane section of $X$.
\item
\label{Bt2}
$X \cap H$ is reduced if $\dim X \geq 1$.
\item
\label{Bt3}
$X \cap H$ is irreducible if $\dim X \geq 2$.
\item
\label{Bt4}
$\dim \sing X \cap H =
\begin{cases}
\dim \sing X - 1 & \text{if} \quad \dim \sing X \geq 1,\\
- 1              & \text{if} \quad \dim \sing X \leq 0.
\end{cases}$
\item
\label{Bt5}
$\deg X \cap H = \deg X$ if $\dim X \geq 1$.
\end{enumerate}

We need the following version of Bertini's Theorem. A proof of this result
can be obtained as a consequence of \cite[Cor.  6.11, p. 89]{Jouanolou1} and 
\cite[Lemma 4.1]{Skoro}. 

\begin{lemma}
\label{Bertini}
Let $X$ be a closed subvariety in $\ProjCan$. Then $\EuScript{U}_{1}(X)$
contains a non\-empty Zariski open set of $\ProjDual$.
\hfill \qedbox
\end{lemma}

Let $r \in \NN$ with $0 \leq r \leq N$. We denote by $\mathcan{G}_{r,N}$ the
Grassmannian of linear varieties of dimension $r$ in $\ProjCan$. Let
$\EuScript{F}_{r}(\ProjCan)$ be the projective variety consisting of
sequences
$$
(E_{1}, \dots, E_{r}) \in \mathcan{G}_{N - 1,N} \times \dots \times
\mathcan{G}_{N - r,N}
$$
making up a \emph{flag} of length $r$:
$$
\ProjCan = E_{0} \supset E_{1} \supset \dots \supset E_{r}
$$
with $\codim E_{m} = m$ for $0 \leq m \leq r$, so that $E_{m}$ is a
hyperplane of $E_{m - 1}$. Let $X$ be a subvariety of $\ProjCan$ and
$(E_{1}, \dots, E_{r})$ a flag in $\EuScript{F}_{r}(\ProjCan)$. We associate
to these data a descending chain of schemes
$X_{0} \supset X_{1} \supset \dots \supset X_{r}$
defined by
$$
X = X_{0}, \qquad X_{m} = X_{m - 1} \cap E_{m} \qquad (1 \leq m \leq r).
$$
Note that $X_{m} = X \cap E_{m}$ for $0 \leq m \leq r$. We say that $(E_{1},
\dots, E_{r})$ is a \emph{regular flag} of length $r$ for $X$ if the
following
conditions hold for $1 \leq m \leq r$:
\begin{enumerate}
\item
\label{Rf1}
$X_{m}$ is a proper linear section of $X$
(and hence, $\dim X_{m} = \dim X - m$ if $\dim X \geq m$ and $X_{m}$ is
empty otherwise).
\item
\label{Rf2}
$X_{m}$ is reduced if $\dim X \geq m$.
\item
\label{Rf3}
$X_{m}$ is irreducible if $\dim X \geq m + 1$.
\item
\label{Rf5}
$
\dim \sing X_{m} =
\begin{cases}
\dim \sing X - m & \text{if} \quad \dim \sing X \geq m,\\
- 1              & \text{if} \quad \dim \sing X \leq m - 1.
\end{cases}
$
\item
\label{Rf4}
$\deg X_{m} = \deg X$ if $\dim X \geq m$.
\end{enumerate}
If $\dim X \geq r$, we denote by $\EuScript{U}_{r}(X)$ the set of $E \in
\mathcan{G}_{N - r,N}$ such that there exists a regular flag $(E_{1},
\dots, E_{r})$ for $X$ with $E_{r} = E$.

By lemma \ref{Bertini} and induction we get:

\begin{proposition}
\label{BertiniFlag}
Let $X$ be a closed subvariety in $\ProjCan$. Then $\EuScript{U}_{r}(X)$
contains a nonempty Zariski open set of $\mathcan{G}_{N - r,N}$.
\hfill \qedbox
\end{proposition}

\begin{corollary}
\label{Equiv}
Let $X$ be a closed subvariety in $\ProjCan$ of dimension $n$, and let
$s \in \NN$ with $0 \leq s \leq n - 2$. Then the following conditions are
equivalent:
\begin{enumerate}
\item
\label{Equiv1}
There is a proper linear section of codimension $s + 1$ of $X$ which is a
nonsingular variety.
\item
\label{Equiv2}
$\dim \sing X \leq s$.
\hfill \qedbox
\end{enumerate}
In particular, $X$ is regular in codimension one if and only if there is a
nonsingular proper linear section $Y$ of dimension $1$ of $X$.
\hfill \qed
\end{corollary}

Following Weil \cite[p. 118]{Weil54}, we call $Y$ a \emph{typical curve
on $X$} if $X$ and $Y$ satisfy the above conditions.

\section{Weak Lefschetz Theorem for Singular Varieties}
\label{WeakLefschetz}

We assume now that $k$ is a perfect field of characteristic $p \geq 0$. Let
$\ell \neq p$ be a prime number, and denote by ${\QQ}_{\ell}$ the field of
$\ell$-adic numbers. Given an algebraic variety $X$ defined
over $k$, by $H^{i}(\bar{X}, \QQ_{\ell})$ we denote the
\emph{\'etale $\ell$-adic cohomology space} of $\bar{X}$ and by
$H^{i}_{c}(\bar{X}, {\QQ}_{\ell})$ the corresponding cohomology spaces with
compact support. We refer to the book of Milne \cite{Milne} and to the
survey of Katz \cite{Katz2} for the definitions and the fundamental theorems
on this theory. These are finite dimensional vector spaces over
${\QQ}_{\ell}$, and they vanish for $i < 0$ as well as for $i > 2 \dim X$.
If $X$ is proper the two cohomology spaces coincide. The \emph{$\ell$-adic
Betti numbers} of $X$ are
$$ b_{i,\ell}(\bar{X}) = \dim H^{i}_{c}(\bar{X},{\QQ}_{\ell})
\quad (0 \leq i \leq 2n).
$$
If $X$ is a nonsingular projective variety, these numbers are independent
of the choice of $\ell$ \cite[p. 27]{Katz2}, and in this case we set
$b_{i}(\bar{X}) = b_{i,\ell}(\bar{X})$. It is conjectured that this is
true for any separated scheme $X$ of finite type.

These spaces are endowed with an action of the Galois group $\mathbf{g} =
\Gal(\overline{k}/k)$. We call a map of such spaces
\emph{$\mathbf{g}$-equivariant} if it commutes with the action of
$\mathbf{g}$. For $c \in \ZZ$, we can consider the \emph{Tate twist} by
$c$ \cite[pp. 163-164]{Milne}. Accordingly, by
$$
H^{i}(\bar{X},{\QQ}_{\ell}(c)) =
H^{i}(\bar{X},{\QQ}_{\ell}) \otimes \QQ_{\ell}(c)
$$
we shall denote the corresponding twisted copy of
$H^{i}(\bar{X},{\QQ}_{\ell})$.

In this section, we shall prove a generalization to singular varieties of the
classical \emph{Weak Lefschetz Theorem} for cohomology spaces of high degree
(cf. \cite[Thm. 7.1, p. 253]{Milne}, \cite[Thm. 7.1, p.  318]{Jouanolou2}).
It seems worthwhile to first review the case of nonsingular varieties, 
which is discussed below. 

\begin{theorem}
\label{LefschetzSGA}
Let $X$ be an irreducible projective scheme of dimension $n$, and
$Y$ a proper linear section of codimension $r$ in $X$ which is a nonsingular
variety. Assume that both $X$ and $Y$ are defined over $k$. Then for each $i \geq n + r$, the closed immersion $\iota : Y
\longrightarrow X$ induces a canonical $\mathbf{g}$-equivariant linear map
$$
\iota_{*} : H^{i - 2r}(\bar{Y},\QQ_{\ell}(- r)) \longrightarrow
H^{i}(\bar{X},\QQ_{\ell}),
$$
called the \emph{Gysin map}, which is an isomorphism for $i \geq n + r + 1$
and a surjection for $i = n + r$.
\end{theorem}

\begin{proof}
There is a diagram
$$
\begin{CD} Y & @>{\iota}>> & X & @<<< & U
\end{CD}
$$
where $U = X \setminus Y$. The corresponding long exact sequence in cohomology with
support in $Y$ \cite[Prop. 1.25, p. 92]{Milne} will be as follows.
\begin{equation*}
\label{support}
\cdots \longrightarrow H^{i - 1}(\bar{U},\QQ_{\ell})
\longrightarrow H^{i}_{\bar{Y}}(\bar{X}, \QQ_{\ell})
\longrightarrow H^{i}(\bar{X}, \QQ_{\ell}) \longrightarrow
H^{i}(\bar{U}, \QQ_{\ell}) \longrightarrow \cdots
\end{equation*}
Since $U$ is a scheme of finite type of dimension $n$ which is the union
of $r$ affine schemes, we deduce from Lefschetz Theorem on the cohomological
dimension of affine schemes \cite[Thm. 7.2, p. 253]{Milne} that
$H^{i}(\bar{U},\QQ_{\ell}) = 0$ if $i \geq n + r$.
Thus, the preceding exact sequence induces a surjection
$$
H^{n + r}_{\bar{Y}}(\bar{X}, \QQ_{\ell}) \longrightarrow
H^{n + r}(\bar{X}, \QQ_{\ell}) \longrightarrow 0,
$$
and isomorphisms
$$
H^{i}_{\bar{Y}}(\bar{X}, \QQ_{\ell}) \isom
H^{i}(\bar{X}, \QQ_{\ell}) \quad (i \geq n + r + 1).
$$
The cohomology groups with support in $\bar{Y}$ can be calculated by
excision in an \'etale neighbourhood of $\bar{Y}$ \cite[Prop. 1.27, p.
92]{Milne}. Let $X' = \reg X$ be the smooth locus of $X$. Since $Y$ is
nonsingular, we find $E \cap \sing X = \emptyset$ by Lemma
\ref{IneqSing}\eqref{IneqSing1}, and hence, $Y \subset X'$, that is, X is
nonsingular in a neighbourhood of $Y$. Thus we obtain isomorphisms
$$
H^{i}_{\bar{Y}}(\bar{X}, \QQ_{\ell}) \isom H^{i}_{\bar{Y}}(\bar{X'},
\QQ_{\ell}), \quad
\textrm{for all }i \geq 0. 
$$
Now $(Y, X')$ is a \emph{smooth pair of $k$-varieties} of codimension $r$ as
defined in \cite[VI.5, p. 241]{Milne}. By the Cohomological Purity
Theorem \cite[Thm. 5.1, p. 241]{Milne}, there are canonical isomorphisms
\begin{equation*}
H^{i - 2 r}(\bar{Y}, \QQ_{\ell}(- r)) \isom H^{i}_{Y}(\bar{X'}, \QQ_{\ell}), \quad 
\textrm{for all }i \geq 0. 
\end{equation*}
This yields the desired results. 
\end{proof}

\begin{corollary}
\label{BettiNb}
Let $X$ be a subvariety of dimension $n$ of $\ProjCan$ with $\dim \sing X
\leq s$ and let $Y$ be a proper linear section of codimension $s + 1 \leq n
- 1$ of $X$ which is a nonsingular variety. Then:
\begin{enumerate}
\item
\label{BettiNb1}
$b_{i,\ell}(\bar{X}) = b_{i - 2s - 2}(\bar{Y})$ \text{if } $i \geq n + s
+ 2$.
\item
\label{BettiNb2}
$b_{n + s + 1,\ell}(\bar{X}) \leq b_{n - s - 1}(\bar{Y})$.
\hfill \qed
\end{enumerate}
\end{corollary}

\begin{remark}
In view of Corollary \ref{Equiv}, relation \eqref{BettiNb1} implies that the
Betti numbers $b_{i,\ell}(\bar{X})$ are independent of $\ell$ for $i \geq n
+ s + 2$.
\end{remark}

Now let $X$ be a irreducible closed subscheme of dimension $n$ of
$\ProjCan$. If $0 \leq r \leq n$, let
$$
\ProjCan = E_{0} \supset E_{1} \supset \dots \supset E_{r}
$$
be a flag in $\EuScript{F}_{r}(\ProjCan)$ and
$$
X = X_{0}, \qquad X_{m} = X_{m - 1} \cap E_{m} \qquad (1 \leq m \leq r)
$$
the associated chain of schemes. We say that $(E_{1}, \dots, E_{r})$ is a
\emph{semi-regular flag} of length $r$ for $X$ if the schemes $X_{1},
\dots, X_{r - 1}$ are irreducible. We denote by $\EuScript{V}_{r}(X)$ the
set of $E \in \mathcan{G}_{N - r,N}$ such that there exists a semi-regular
flag for $X$ with $E_{r} = E$. Since $\EuScript{U}_{r}(X) \subset
\EuScript{V}_{r}(X)$, the set $\EuScript{V}_{r}(X)$ contains a nonempty
Zariski open set in $\mathcan{G}_{N - r,N}$. A \emph{semi-regular pair} is a
couple $(X, Y)$ where $X$ is an irreducible closed subscheme of $\ProjCan$
and $Y$ is a proper linear section $Y = X \cap E$ of codimension $r$ in $X$,
with $E \in \EuScript{V}_{r}(X)$. Hence, if $r = 1$, a semi-regular pair is
just a couple $(X, X \cap H)$ where $X$ is irreducible and $X \cap H$ is a
proper hyperplane section of $X$.

The generalization to singular varieties of Theorem \ref{LefschetzSGA} that
we had alluded to in the beginning of this section is the following.

\begin{theorem}
\label{LefschetzHigh}
\emph{(General Weak Lefschetz Theorem, high degrees)}.
Let $(X, Y)$ be a semi-regular pair with $\dim X = n$, $Y$ of codimension
$r$ in $X$ and $\dim \sing Y = \sigma$. 
Assume that both $X$ and $Y$ are defined over $k$.
Then for each $i \geq n + r + \sigma
+ 1$ there is a canonical $\mathbf{g}$-equivariant linear map
$$
\iota_{*} : H^{i - 2r}(\bar{Y},\QQ_{\ell}(- r)) \longrightarrow
H^{i}(\bar{X},\QQ_{\ell})
$$
which is an isomorphism for $i \geq n + r + \sigma + 2$ and a surjection for
$i = n + r + \sigma + 1$.
If $X$ and $Y$ are nonsingular, and if there is a regular flag $(E_{1},
\dots, E_{r})$ defined over $k$ for $X$ with $E_{r} \cap X = Y$, then $\iota_{*}$ is the
Gysin map induced by the immersion $\iota : Y \longrightarrow X$.
\end{theorem}

We call $\iota_{*}$ the \emph{Gysin map}, since it generalizes the classical
one when $X$ and $Y$ are nonsingular.

\begin{proof}
For hyperplane sections, \idest \ if $r = 1$, this is a result of
Skorobogatov \cite[Cor. 2.2]{Skoro}. Namely, he proved that if $X \cap H$ is
a proper hyperplane section of $X$ and if
$$\alpha = \dim X + \dim \sing (X \cap H),$$
then, for each $i \geq 0$, there is a $\mathbf{g}$-equivariant linear map
\begin{equation}
\label{LH1}
H^{\alpha + i}(\bar{X} \cap \bar{H},\QQ_{\ell}(- 1)) \longrightarrow
H^{\alpha + i + 2}(\bar{X},\QQ_{\ell})
\end{equation}
which is a surjection for $i = 0$ and an isomorphism for $i > 0$. Moreover
he proved also that if $X$ and $X \cap H$ are nonsingular then $\iota_{*}$ is
the Gysin map. In the general case we proceed by iteration. Let
$(E_{1}, \dots, E_{r})$ be a semi-regular flag of
$\EuScript{F}_{r}(\ProjCan)$ such that $Y = X \cap E_{r}$ and let
$$
X = X_{0}, \qquad X_{m} = X_{m - 1} \cap E_{m} \qquad (1 \leq m \leq r).
$$
be the associated chain of schemes. First,
\begin{equation}
\label{LH2}
\dim X_{m} = n - m \quad \text{for} \quad 1 \leq m \leq r.
\end{equation}
In fact, let $\eta_{m} = \dim X_{m - 1} - \dim X_{m}$. Then
$0 \leq \eta_{m} \leq 1$ by Krull's Principal Ideal Theorem. But
$$\eta_{1} + \dots + \eta_{r} = \dim X_{0} - \dim X_{r} = r.$$
Hence $\eta_{m} = 1$ for $1 \leq m \leq r$, which proves \eqref{LH2}. This
relation implies that $X_{m}$ is a proper hyperplane section of $X_{m - 1}$.
Since $X_{m - 1}$ is irreducible by hypothesis, the couple $(X_{m - 1},
X_{m})$ is a semi-regular pair for $1 \leq m \leq r$ and we can apply the
Theorem in the case of codimension one. From Lemma \ref{IneqSing} we deduce
$$
\dim X_{m} - \dim \sing X_{m} \geq \dim Y - \dim \sing Y = n - r -
\sigma,
$$
and hence, $\dim \sing X_{m} \leq r - m + \sigma$. Then
$$\dim X_{m - 1} + \dim \sing X_{m} \leq \alpha(m),$$
where $\alpha(m) = n + r - 2 m + \sigma + 1$. We observe that $\alpha(m - 1)
= \alpha(m) + 2$. Hence from \eqref{LH1} we get a map
$$
H^{\alpha(m) + i}(\bar{X}_{m},\QQ_{\ell}(- m)) \longrightarrow
H^{\alpha(m - 1) + i}(\bar{X}_{m - 1},\QQ_{\ell}(- m + 1))
$$
which is a surjection for $i = 0$ and an isomorphism for $i > 0$.
The composition of these maps gives a map
$$
\iota_{*} : H^{\alpha(r) + i}(\bar{Y},\QQ_{\ell}(- r)) \longrightarrow
H^{\alpha(0) + i}(\bar{X},\QQ_{\ell})
$$
which the same properties. Since
$$\alpha(r) = n - r + \sigma + 1, \quad \alpha(0) = n + r + \sigma + 1,$$
Substituting $j = \alpha(0) + i = \alpha(r) + 2 r + i$, we get
$$
\iota_{*} : H^{j - 2r}(\bar{Y},\QQ_{\ell}(- r)) \longrightarrow
H^{j}(\bar{X},\QQ_{\ell}),
$$
fulfilling the required properties. Now recall from \cite[Prop. 6.5(b),
p. 250]{Milne} that if
$$
X_{2} \stackrel{\iota_{2}}{\longrightarrow} X_{1}
\stackrel{\iota_{1}}{\longrightarrow} X_{0}
$$
is a so-called \emph{smooth triple} over $k$, then the Gysin map for
$\iota_{1} \circ \iota_{2}$ is the composition of the Gysin maps for
$\iota_{1}$ and $\iota_{2}$, which proves the last assertion of the
Theorem.
\end{proof}

The following Proposition gives a criterion for a pair to be semi-regular.

\begin{proposition}
\label{LHReg1}
Let $X$ be an irreducible closed subscheme of $\ProjCan$ and  $Y$ a
proper linear section of $X$ of dimension $\geq 1$ which is regular in
codimension one. Assume that both $X$ and $Y$ are defined over $k$.
Then $(X, Y)$ is a semi-regular pair. Moreover $Y$ is
irreducible.
\end{proposition}

\begin{proof}
As in the proof of Theorem \ref{LefschetzHigh}, we first prove the
proposition if $Y$ is a regular hyperplane section of $X$. In that case
$(X, Y)$ is a semi-regular pair, as we pointed out, and the only statement
to prove is that $Y$ is irreducible. If $Y$ is regular in codimension one,
then
$$\sigma = \dim \sing Y \leq \dim Y - 2 \leq n - 3.$$ This implies that $2n
\geq n + \sigma + 3$ and we can take $i = 2n$ in Theorem \ref{LefschetzHigh}
in order to get an isomorphism
$$H^{2n - 2}(\bar{Y},\QQ_{\ell}(- 1)) \isom H^{2n}(\bar{X},\QQ_{\ell}).$$
Now for any scheme $X$ of dimension $n$, the dimension of
$H^{2n}(\bar{X},\QQ_{\ell})$ is equal to the number of irreducible
components of $X$ of dimension $n$, as can easily be deduced from the
Mayer-Vietoris sequence \cite[Ex. 2.24, p. 110]{Milne}. Since $Y$ is a
proper linear section, it is equidimensional and the irreducibility of $Y$
follows from the irreducibility of $X$. The general case follows by
iteration.
\end{proof}

Observe that Theorem \ref{LefschetzSGA} is an immediate consequence of
Theorem \ref{LefschetzHigh} and Proposition \ref{LHReg1}.

\begin{remark}[Weak Lefschetz Theorem, low degrees]
\label{BigLefschtezLow}
We would like to point out the following result, although we shall not use
it. Assume that the resolution of singularities is possible, that is, the
condition $\resol$ stated in section \ref{CentralBettiSingular} below,
holds.
Let $X$ be a closed subscheme of dimension $n$ in $\ProjCan$ defined over $k$, and let there
be given a diagram
$$
\begin{CD} Y & @>{\iota}>> & X & @<<< & U
\end{CD}
$$
where $\iota$ is the closed immersion of a proper linear section $Y$ of
$X$ of codimension $r$ defined over $k$, and where $U = X \setminus Y$.
If $U$ is a local complete intersection, then the canonical $\mathbf{g}$-equivariant
linear map
$$
\iota^{*} : H^{i}(\bar{X},\QQ_{\ell}) \longrightarrow
H^{i}(\bar{Y},\QQ_{\ell})
$$
is an isomorphism if $i \leq n - r - 1$ and an injection if $i = n - r$.

This theorem is a consequence of the Global Lefschetz Theorem of
Grothendieck for low degrees \cite[Cor. 5.7 p. 280]{Raynaud}. Notice that no
hypotheses of regularity are put on $Y$ in that statement. Moreover, if $U$
is nonsingular, the conclusions of the theorem are valid without assuming
condition $\resol$ \cite[Thm. 7.1, p. 318]{Jouanolou2}.
\end{remark}

\begin{remark}[Poincar\'e Duality]
\label{PoincDual}
By combining Weak Lefschetz Theorem for high and low degrees, we get a weak
version of Poincar\'e Duality for singular varieties. Namely, assume that
$\resol$ holds and let $X$ be a closed subvariety of dimension $n$ in
$\ProjCan$ defined over $k$ which is a local complete intersection such that
$\dim \sing X \leq s$. Assume that there is a proper linear section
of codimension $s + 1$ of $X$ defined over $k$. Then, for $0 \leq i \leq n - s - 2$ there is a
nondegenerate pairing
$$ H^{i}(\bar{X},\QQ_{\ell}) \times H^{2n - i}(\bar{X},\QQ_{\ell}(n))
\longrightarrow \QQ_{\ell}.
$$
Furthermore, if we denote this pairing by $(\xi, \eta)$, then
$$
(g.\xi, g.\eta) =
(\xi, \eta) \quad \text{for every } g \in \mathbf{g}.
$$
In particular,
$$ b_{i, \ell}(\bar{X}) = b_{2n - i}(\bar{X}) \quad
\text{for} \quad 0 \leq i \leq n - s - 2,
$$
and these numbers are independent of $\ell$.
\end{remark}

\section{Cohomology of Complete Intersections}
\label{Cohci}

Let $k$ be a perfect field. A closed subscheme $X$ of $\ProjCan = \ProjCan_{k}$
of codimension
$r$ is a \emph{complete intersection} if $X$ is the closed subscheme
determined by an ideal $\mathfrak{I}$ generated by $r$ homogeneous
polynomials $f_{1}, \dots, f_{r}$.

A complete intersection is a local complete intersection, and in particular
$X$ is Cohen-Macaulay and equidimensional. Moreover, if $\dim X \geq 1$,
then $X$ is connected, hence $X$ is integral if it is regular in codimension
one.

The multidegree $\mathbf{d} = (d_{1}, \dots, d_{r})$ of the system
$f_{1}, \dots, f_{r}$, usually labelled so that $d_{1} \geq \dots \geq
d_{r}$, depends only on $\mathfrak{I}$ and not of the chosen system of
generators $f_{1},\dots, f_{r}$, since the \emph{Hilbert series} of the
homogeneous coordinate ring of $X$ equals
$$
H(T) =
\frac{(1 - T^{d_{1}}) (1 - T^{d_{2}}) \cdots
(1 - T^{d_{r}})}{(1 - T)^{N + 1}}
$$
\cite[Ex. 7.15, p. 350]{Northcott} or \cite[Prop. 6, p. AC VIII.50]{BkiAC8}.
This implies
$$\deg X = d_{1} \cdots d_{r}.$$
First of all, recall that if $X = \PP^{n}$, and if $0 \leq i \leq 2n$, then
\cite[Ex. 5.6, p. 245]{Milne}:
\begin{equation}
\label{HomZero}
H^{i}(\bar{X},\QQ_{\ell}) =
\begin{cases}
\QQ_{\ell}(-i/2) & \text{if $i$ is even}\\
0                & \text{if $i$ is odd}.
\end{cases}
\end{equation}
Consequently, $\dim H^{i}(\PP^{n}_{\bar{k}},\QQ_{\ell}) = \varepsilon_{i}$ for
$0
\leq i \leq 2n$, where we set
\begin{equation*}
\label{varepsilon}
\varepsilon_{i} =
\begin{cases}
1 & \text{if $i$ is even} \\
0 & \text{if $i$ is odd}.
\end{cases}
\end{equation*}
Now let $X$ be a nonsingular projective subvariety of $\ProjCan$ of
dimension $n \geq 1$. For any $i \geq 0$, the image of the canonical morphism
$H^{i}(\ProjCan_{\bar{k}},{\QQ}_{\ell}) \longrightarrow
H^{i}(\bar{X},\QQ_{\ell})$ is isomorphic to
$H^{i}(\PP^{n}_{\bar{k}},\QQ_{\ell})$. Hence, one obtains a short exact
sequence
\begin{equation}
\label{ExactSeq}
0 \longrightarrow H^{i}(\PP^{n}_{\bar{k}},\QQ_{\ell}) \longrightarrow
H^{i}(\bar{X},\QQ_{\ell}) \longrightarrow P^{i}(\bar{X},\QQ_{\ell})
\longrightarrow 0
\end{equation}
where $P^{i}(\bar{X},\QQ_{\ell})$ is the cokernel of the canonical morphism,
called the \emph{primitive part} of $H^{i}(\bar{X},\QQ_{\ell})$. If $i$ is
even, the image of $H^{i}(\PP^{n}_{\bar{k}},\QQ_{\ell})$ is the
one-dimensional vector space generated by the cup-power of order $i/2$ of the
cohomology class of a hyperplane section. Hence, if we define the
\emph{primitive $i$-th Betti number} of a \emph{nonsingular} projective
variety $X$ over $k$ as
\begin{equation*}
\label{bprime}
b_{i,\ell}'(X) = \dim P^{i}(\bar{X},\QQ_{\ell}),
\end{equation*}
then by \eqref{HomZero} and \eqref{ExactSeq}, we have
\begin{equation*}
\label{epsilon}
b_{i,\ell}(X) = b'_{i,\ell}(X) + \varepsilon_{i}.
\end{equation*}

\begin{proposition}
\label{BettiIC}
Let $X$ be a nonsingular complete intersection in $\ProjCan$ of codimension
$r$, of multidegree $\mathbf{d}$, and let $\dim X = n = N - r$.
\begin{enumerate}
\item
\label{BettiIC1}
If $i \neq n$, then $P^{i}(\bar{X},\QQ_{\ell}) = 0$ for $0 \leq i \leq 2n$.
Consequently, $X$ satisfies \eqref{HomZero} for these values of $i$.
\item
\label{BettiIC2}
The $n$-th Betti number of $X$ depends only on $n$, $N$ and $\mathbf{d}$.
\end{enumerate}
\end{proposition}

\begin{proof}
Statement \eqref{BettiIC1} is easily proved by induction on $r$, if we use
the Veronese embedding, Weak Lefschetz Theorem \ref{LefschetzHigh} and
Poincar\'e Duality in the case of nonsingular varieties. Statement
\eqref{BettiIC2} follows from Theorem \ref{ICFormula} below.
\end{proof}

In view of Proposition \ref{BettiIC}, we denote by $b'_{n}(N,
\mathbf{d})$ the primitive $n$-th Betti number of any nonsingular complete
intersection in $\ProjCan$ of codimension $r = N - n$ and of multidegree
$\mathbf{d}$. It will be described explicitly in Section \ref{CentralBetti}.

The cohomology in lower degrees of general (possibly singular) complete
intersections can be calculated with the help of the following simple
result.

\begin{proposition}
\label{LefschtezLow}
Let $X$ be a closed subscheme of $\ProjCan$ defined by the vanishing of
$r$ forms defined over $k$. Then the $\mathbf{g}$-equivariant restriction map
\begin{equation*}
\label{LowCanon}
\iota^{*} : H^{i}(\ProjCan_{\bar k},\QQ_{\ell}) \longrightarrow
H^{i}(X,\QQ_{\ell})
\end{equation*}
is an isomorphism for $i \leq N - r - 1$, and is injective for $i = N - r$.
\end{proposition}

\begin{proof}
The scheme $U = \ProjCan \setminus X$ is a scheme of finite type of dimension
$N$ which is the union of $r$ affine open sets. From Affine Lefschetz Theorem
\cite[Thm. 7.2, p. 253]{Milne}, we deduce
$$H^{i}(\bar{U},\QQ_{\ell}) = 0 \quad \text{if} \ i \geq N + r.$$
Since $U$ is smooth, by Poincar\'e Duality we get
$$H^{i}_{c}(\bar{U},\QQ_{\ell}) = 0 \quad \text{if} \ i \leq N - r.$$
Now, the excision long exact sequence in compact cohomology
\cite[Rem. 1.30, p. 94]{Milne} gives:
\begin{equation}
\label{excision}
\cdots \longrightarrow H^{i}_{c}(\bar{U},\QQ_{\ell})
\longrightarrow H^{i}(\ProjCan, \QQ_{\ell})
\longrightarrow H^{i}(\bar{X}, \QQ_{\ell}) \longrightarrow
H^{i + 1}_{c}(\bar{U}, \QQ_{\ell}) \longrightarrow \cdots
\end{equation}
from which the result follows.
\end{proof}

We shall now study the cohomology in higher degrees of general complete
intersections.

\begin{proposition}
\label{Bsci}
Let $X$ be a complete intersection in $\ProjCan$ of dimension $n \geq 1$
and multidegree $\mathbf{d}$ with $\dim \sing X \leq s$. Then:
\begin{enumerate}
\item
\label{Bsci1}
Relation \eqref{HomZero} holds, and hence, $b'_{i}(X) = 0$, if $n + s + 2
\leq i \leq 2n$.
\item
\label{Bsci2}
$b_{n + s + 1,\ell}(X) \leq b_{n - s - 1}(N - s - 1, \mathbf{d}).$
\item
\label{Bsci3}
Relation \eqref{HomZero} holds, and hence, $b'_{i}(X) = 0$,
if $0 \leq i \leq n - 1$.
\end{enumerate}
\end{proposition}

\begin{proof}
Thanks to Proposition \ref{BertiniFlag}, there are regular flags for $X$;
let $(E_{1}, \dots, E_{s + 1})$ be one of them, and let $Y = X \cap E_{s +
1}$. Then $Y$ is a nonsingular variety by definition of regular flags, and a
complete intersection. Applying the Weak Lefschetz Theorem
\ref{LefschetzSGA}, we deduce that the Gysin map
$$
\iota_{*} : H^{i - 2s - 2}(\bar{Y},\QQ_{\ell}(- s - 1)) \longrightarrow
H^{i}(\bar{X},\QQ_{\ell})
$$
is an isomorphism for $i \geq n + s + 2$ and a surjection for $i = n + s +
1$, which proves \eqref{Bsci1} and \eqref{Bsci2} in view of Proposition
\ref{BettiIC}. Assertion \eqref{Bsci3} follows from Proposition
\ref{LefschtezLow}.
\end{proof}

\begin{remarks}
\label{ThreeRemarks}
(i) The case $s = 0$ of this theorem is a result of I.E. Shparlinski\u{\i}
and A.N. Skorobogatov \cite[Thm. 2.3]{Sh-S}, but the proof of (a) and (b) in
that Theorem is unclear to us: in any case their proof use results which are
valid only if the characteristic is $0$ or if one assumes the resolution of
singularities (condition $\resol$ in section \ref{CentralBettiSingular}
below).\\
(ii) In \cite[proof of Thm. 1]{Katz1}, N. Katz proves \eqref{Bsci1} with the
arguments used here to prove Proposition \ref{BsciRnp} below.
\end{remarks}

\begin{remark}
\label{Even}
For further reference, we note that if $s \geq 0$ and if $i = n + s + 1$ is
even, then $H^{i}(X,\QQ_{\ell})$ contains a subspace isomorphic to
$\QQ_{\ell}(- i/2)$. In fact, with the notations of the proof of Proposition
\ref{Bsci}, the map
$$
\iota_{*} : H^{n - s - 1}(\bar{Y},\QQ_{\ell}(- s - 1)) \longrightarrow
H^{n + s + 1}(\bar{X},\QQ_{\ell})
$$
is a surjection.
\end{remark}

\section{The Central Betti Number of Complete Intersections}
\label{CentralBetti}

We now state a well-known consequence of the Riemann-Roch-Hirzebruch Theorem
\cite[Satz 2.4, p. 136]{Hirzebruch}, \cite[Cor. 7.5]{Jouanolou2}; see
\cite[Cor 4.18]{AdolSper2} for an alternative proof. If $\nu = (\nu_{1},
\dots, \nu_{r}) \in \NN^{r}$, and if
$\mathbf{d} = (d_{1},\dots, d_{r}) \in (\NN^{\times})^{r}$, we define
$$\mathbf{d}^{\nu} = d_{1}^{\nu_{1}} \dots d_{r}^{\nu_{r}}.$$
If $c \geq 1$, let
$$
M(c) = \set{(\nu_{1}, \dots, \nu_{r}) \in \NN^{r}}
{\nu_{1} + \dots + \nu_{r} = c
\ \text{and} \ \nu_{i} \geq 1 \ \text{for} \ 1 \leq i \leq r}.
$$

\begin{theorem}
\label{ICFormula}
The primitive $n$-th Betti number of any nonsingular complete intersection
in $\ProjCan$ of codimension $r = N - n$ and of multidegree $\mathbf{d}$ is
\begin{equation*}
\label{BettiForm}
b'_{n}(N, \mathbf{d}) = (- 1)^{N - r + 1} (N - r + 1) +
(- 1)^{N} \sum_{c = r}^{N} (- 1)^{c} \binom{N + 1}{c + 1}
\sum_{\mathbf{\nu} \in M(c)} \ \mathbf{d}^{\mathbf{\nu}}.
\rlap \qedbox
\end{equation*}
\end{theorem}

We now give estimates on the primitive $n$-th Betti number of a nonsingular
complete intersection.

\begin{proposition}
\label{BettiNbBound}
Let $\mathbf{d} = (d_{1},\dots, d_{r}) \in (\NN^{\times})^{r}$ and
$b'_{n}(N, \mathbf{d})$ be the $n^{th}$ primitive Betti number of a
nonsingular complete intersection in $\ProjCan$ of dimension $n = N - r$ and
multidegree $\mathbf{d}$. Let $\delta = \max(d_{1}, \dots, d_{r})$ and
$d = \deg X = d_{1} \cdots d_{r}$. Then, for $r \geq 1$, we have
$$
b'_{n}(N, \mathbf{d}) \leq
(- 1)^{n + 1} (n + 1) + d \binom{N + 1}{n} (\delta + 1)^{n}.
$$
In particular,
$$
b'_{n}(N, \mathbf{d}) \leq \binom{N + 1}{n} (\delta + 1)^{N}.
$$
\end{proposition}

\begin{proof}
Observe that for any $c \geq r$ and $\nu \in M(c)$, we have
$$
\mathbf{d}^{\mathbf{\nu}} \leq d \delta^{c - r}
\quad \text{and} \quad \card{M(c)} = \binom{c - 1}{r - 1}.
$$
Hence, by Theorem \ref{ICFormula},
$$
b'_{n}(N, \mathbf{d}) - (- 1)^{n + 1} (n + 1) \leq
\sum_{c = r}^{N}
\binom{N + 1}{c + 1} \binom{c - 1}{r - 1} d \delta^{c - r}.
$$
Now, for $c \geq r \geq 1$, we have
$$
\binom{N + 1}{c + 1} = \dfrac{N(N + 1)}{c(c + 1)} \binom{N - 1}{c - 1}
\leq \dfrac{N(N + 1)}{r(r + 1)} \binom{N - 1}{c - 1},
$$
and moreover,
$$
\sum_{c = r}^{N}
\binom{N - 1}{c - 1} \binom{c - 1}{r - 1} \delta^{c - r} =
\binom{N - 1}{r - 1} \sum_{c = r}^{N} \binom{n}{c - r} \delta^{c - r} =
\binom{N - 1}{r - 1} (\delta + 1) ^{n}.
$$
It follows that
$$
b'_{n}(N, \mathbf{d}) - (- 1)^{n + 1} (n + 1) \leq
d \, \dfrac{N(N + 1)}{r(r + 1)} \binom{N - 1}{r - 1} (\delta + 1) ^{n} =
d \binom{N + 1}{r + 1} (\delta + 1) ^{n}.
$$
This proves the first assertion. The second assertion is trivially satisfied
if $n = 0$, while if $n \geq 1$, then
$$\smallbinom{N + 1}{n} \geq N + 1 \geq n + 1,$$
and since $d = d_{1} \cdots d_{r} < (\delta + 1)^{r}$, we have $d (\delta +
1)^{n} \leq (\delta + 1)^{N} - 1$ so that
\begin{eqnarray*}
(- 1)^{n + 1} (n + 1) + d \smallbinom{N + 1}{n} (\delta + 1)^{n}
& \leq &
(n + 1) + (\delta + 1)^{N} \smallbinom{N + 1}{n} - \smallbinom{N + 1}{n}
\\ & \leq &
\smallbinom{N + 1}{n} (\delta + 1)^{N},
\end{eqnarray*}
and the second assertion is thereby proved.
\end{proof}

\begin{examples}
\label{BettiEx}
The calculations in the examples below are left to the reader.

\noindent
(i) The primitive $n$-th Betti number of nonsingular hypersurfaces of
dimension $n$ and degree $d$ is equal to:
$$
b'_{n}(N, d) = \frac{d - 1}{d} ((d - 1)^{N} - (- 1)^{N})
\leq (d - 1)^{N} - \varepsilon_{N}.
$$
In particular, if $N = 2$ (plane curves), then $b_{1}(d) = (d - 1)(d - 2)$,
whereas if $N = 3$ (surfaces in $3$-space), then
$$
b'_{2}(3, d) = (d - 1) ((d - 1)(d - 2) + 1) = d^{3} - 4 d^{2} + 6 d - 3.
$$

\noindent (ii)
The Betti number of a nonsingular curve which is a complete intersection
of $r = N - 1$ forms in $\ProjCan$ of multidegree $\mathbf{d}= (d_{1},
\dots, d_{r})$ is equal to
$$
b_{1}(N, \mathbf{d}) = b'_{1}(N, \mathbf{d}) =
(d_{1} \cdots d_{r})(d_{1} + \dots + d_{r} - N - 1) + 2.
$$
Now observe that if $r$ and $d_{1}, \dots, d_{r}$ are any positive integers,
then
$$d_{1} + \dots + d_{r} \leq (d_{1} \cdots d_{r}) + r - 1,$$
and the equality holds if and only if either $r = 1$ or $r > 1$ and at least
$r - 1$ among the numbers $d_{i}$ are equal to $1$.

To see this, we proceed by induction on $r$. The case $r = 1$ is trivial and
if $r = 2$, then the assertion follows easily from the identity
$$d_{1}d_{2} + 1 - (d_{1} + d_{2}) = (d_{1} - 1)(d_{2} - 1).$$
The inductive step follows readily using the assertion for $r = 2$.

For the nonsingular curve above, if we let $d = d_{1} \cdots d_{r}$, then
$d$ is its degree and
$$
b_{1}(N, \mathbf{d}) = d(d_{1} + \dots + d_{r} - r - 2) + 2
\leq d(d - 3) + 2 = (d - 1)(d - 2).
$$
Since for a nonsingular plane curve of degree $d$ the first Betti number is,
as noted in (i), always equal to $(d - 1)(d - 2)$, it follows that in
general,
$$b_{1}(N, \mathbf{d}) \leq (d - 1)(d - 2)$$ and the equality holds if and
only if the corresponding curve is a nonsingular plane curve.

\noindent (iii)
The Betti number of a nonsingular surface which is a complete intersection
of $r = N - 2$ forms in $\ProjCan$ of multidegree $\mathbf{d} = (d_{1},
\dots, d_{r})$ is equal to
$$
b_{2}(N, \mathbf{d}) = b'_{2}(N, \mathbf{d}) + 1
= d \left( \binom{r + 3}{2} - (r + 3) \sum_{1 \leq i \leq r} d_{i}
+ \sum_{1 \leq i,j \leq r} d_{i} d_{j} \right) - 2,
$$
where $d = d_{1} d_{2} \cdots d_{r}$ is the degree of the surface.

\noindent (iv)
The primitive $n$-th Betti number of a complete intersection defined by
$r = 2$ forms of the same degree $d$ is equal to
$$
b'_{n}(N, (d, d)) =
(N - 1) (d - 1)^{N} + 2 \, \frac{d - 1}{d}((d - 1)^{N - 1} + (- 1)^{N}).
$$
\end{examples}

\section{Zeta Functions and the Trace Formula}
\label{TraceForm}

In this section $k = \Fq$ is the finite field with $q$ elements. We denote
by $k_{r}$ the subfield of $\overline{k}$ which is of degree $r$ over $k$.
Let $X$ be a separated scheme of finite type defined over the field $k$. As
before, we denote by $\bar{X} = X \otimes_{k} \overline{k}$ the scheme deduced
from $X$ by base field extension from $k$ to $\overline{k}$ ; note that
$\bar{X}$ remains unchanged if we replace $k$ by one of its extensions. The
\emph{zeta function} of $X$ is
\begin{equation}
\label{DefZeta}
Z(X, T) =
\exp \sum_{r = 1}^{\infty} \frac{T^{r}}{r} \ \card{X(k_{r})},
\end{equation}
where $T$ is an indeterminate. From the work of Dwork and Grothendieck
(see, for example, the Grothendieck-Lefschetz Trace Formula below), we know
that the function  $Z(X, T)$ is a rational function of $T$. This means
that there are two families of complex numbers $(\alpha_{i})_{i \in I}$ and
$(\beta_{j})_{j \in J}$, where $I$ and $J$ are finite sets, such that
\begin{equation}
\label{DworkFracZeta}
Z(X, T) = \frac{
\displaystyle{\prod_{j \in J}}(1 - \beta_{j} T)}
{\displaystyle{\prod_{i \in I}}(1 - \alpha_{i} T)}.
\end{equation}
We assume that the fraction in the right-hand side of the above equality
is irreducible. Thus, the family $(\alpha_{i})_{i \in I}$ (resp.
$(\beta_{j})_{j \in J}$) is exactly the family of poles (resp. of zeroes) of
$Z(X, T)$, each number being enumerated a number of times equal to its
multiplicity. We call the members of the families $(\alpha_{i})_{i \in I}$
and $(\beta_{j})_{j \in J}$ the \emph{characteristic roots} of $Z(X, T)$.
The
\emph{degree} $\deg Z(X, T)$ of $Z(X, T)$ is the degree of its numerator
minus
the degree of its denominator; the \emph{total degree} $\TotDeg Z(X, T)$ of
$Z(X, T)$ is the sum of the degrees of its numerator and of its denominator.
In the usual way, from the above expression of $Z(X, T)$, we deduce
$$
\card{X(k_{r})} = \sum_{i \in I} \alpha_{i}^{r} - \sum_{j \in J}
\beta_{j}^{r}.
$$
In order to simplify the notation, we write now
$$H^{i}_{c}(\bar{X}) = H^{i}_{c}(\bar{X},\QQ_{\ell}),$$
where $\ell $ is a prime number other than $p$; in the case of proper
subschemes, we may use $H^{i}(\bar{X})$ instead. If we denote by $\varphi$
the element of $\mathbf{g}$ given by $\varphi(x) = x^q$, then the
\emph{geometric Frobenius element} $F$ of $\mathbf{g}$ is defined to be the
inverse of $\varphi$. In the twisted space $\QQ_{\ell}(c)$, we have
$$
F.x = q^{- c} \, x \quad \text{for} \quad x \in \QQ_{\ell}(c).
$$
The geometric Frobenius element $F$ canonically induces on
$H^{i}_{c}(\bar{X})$ an endomorphism denoted by $F \mid H^{i}_{c}(\bar{X})$.
A number $\alpha \in \overline{\QQ_{\ell}}$ is \emph{pure of weight $r$} if
$\alpha$ is an algebraic integer and if $\card{\iota(\alpha)} = q^{r/2}$ for
any embedding $\iota$ of $\overline{\QQ_{\ell}}$ into $\CC$.

We recall the following fundamental results.

\begin{theorem*}
Let $X$ be a separated scheme of finite type over $k$, of dimension $n$.
Then the \emph{Grothendieck-Lefschetz Trace Formula} holds:
\begin{equation}
\label{GrTF}
\card{X(k_{r})} =
\sum_{i = 0}^{2 n}(-1)^{i} \ \Tr(F^{r} \mid H^{i}_{c}(\bar{X})),
\end{equation}
and \emph{Deligne's Main Theorem} holds:
\begin{equation}
\label{DeMT}
\text{The eigenvalues of} \ F \mid H^{i}_{c}(\bar{X})
\text{ are pure of weight} \ \leq i.
\end{equation}
\end{theorem*}

See, for instance, \cite[Thm. 13.4, p. 292]{Milne} for the
Grothendieck-Lefschetz Trace Formula, and \cite[Thm. 1, p. 314]{Deligne2}
for
Deligne's Main Theorem.

The Trace Formula \eqref{GrTF} is equivalent to the equality
\begin{equation}
\label{FracZeta}
Z(X, T) = \dfrac{P_{1,\ell}(X, T) \cdots P_{2n - 1,\ell}(X, T)}
{P_{0,\ell}(X, T) \cdots P_{2n,\ell}(X, T)},
\end{equation}
where $P_{i,\ell}(X, T) = \det(1 - T F \mid H^{i}_{c}(\bar{X}))$. We can
write
\begin{equation}
\label{DefPoly00}
P_{i,\ell}(X, T) = \prod_{j = 1}^{b_{i,\ell}} (1 - \omega_{ij,\ell} T),
\end{equation}
where $\omega_{ij,\ell} \in \overline{\QQ_{\ell}}$, and $b_{i,\ell} =
b_{i,\ell}(\bar{X}) = \deg P_{i,\ell}(X, T)$. The numbers $\omega_{ij,\ell}$
are called the \emph{reciprocal roots} of $P_{i,\ell}(X, T)$. The Trace
Formula \eqref{GrTF} may be written as
$$
\card{X(k)} =
\sum_{i = 0}^{2 n} \ (-1)^{i} \sum_{j = 1}^{b_{i,\ell}} \omega_{ij,\ell},
$$
with the convention that if $b_{i,\ell} = 0$, then the value of the
corresponding sum is zero. The \emph{compact \'etale $\ell$-adic
Euler-Poincar\'e characteristic} of $X$ is
$$
\chi_{\ell}(X) = \sum_{i = 0}^{2n} (-1)^{i} \, b_{i, \ell}(X).
$$
On the other hand, let us define
$$\sigma_{\ell}(X) = \sum_{i = 0}^{2 n} b_{i,\ell}(X).$$
The equality between the right-hand sides of \eqref{DworkFracZeta} and
\eqref{FracZeta} imply
$$\deg Z(X, T) = \chi_{\ell}(X).$$
Hence, the compact \'etale Euler-Poincar\'e characteristic of $X$
is independent of $\ell$, and the number $\deg Z(X, T)$ depends only on
$\bar{X}$. In the same way we find
$$\TotDeg Z(X, T) \leq \sigma_{\ell}(X).$$
Here, the two sides may be different because of a possibility of
cancellations occurring in the right-hand side of \eqref{FracZeta}.

If $k'$ is any field, let us say that a projective scheme $X$ defined over
$k'$ is of \emph{type} $(m, N, \mathbf{d})$ if $X$ is a closed subscheme in
$\ProjCan_{k'}$ which can be defined, scheme-theoretically, by the vanishing
of a system of $m$ nonzero forms with coefficients in $k'$, of multidegree
$\mathbf{d} = (d_{1}, \dots, d_{m})$.

We now state a result of N. Katz \cite[Cor. of Th. 3]{Katz4}, whose proof
is based on a result of Adolphson and Sperber \cite[Th. 5.27]{AdolSper1}.
If one checks the majorations in the proof of \cite[\loccit]{Katz4},
the number $9$ appearing therein can be replaced by the number $8$ in the
inequality below.

\begin{theorem*}[Katz's Inequality]
Let $X$ be a closed subscheme in $\ProjCan$ defined over an algebraically
closed field, and of type $(m, N, \mathbf{d})$. If $\mathbf{d} = (d_{1},
\dots,
d_{m})$, then let $\delta = \max(d_{1}, \dots, d_{m})$. We have
\begin{equation}
\label{KatzIneq}
\sigma_{\ell}(X) \leq 8 \times 2^{m} \times (m \delta + 3)^{N + 1}.
\rlap \qedbox
\end{equation}
\end{theorem*}

We would like to compare the number of points of $X$ and that of the
projective space of dimension equal to that of $X$. Hence, if $\dim X = n$,
we introduce the rational function
\begin{equation}
\label{GenSeries}
\dfrac{Z(X, T)}{Z(\PP^{n}, T)} =
\exp \sum_{r = 1}^{\infty} \frac{T^{r}}{r} \
\left( \ \card{X(k_{r})} - \card{\PP^{n}(k_{r})} \ \right) \, .
\end{equation}
where
$$
Z(\PP^{n}, T) = \dfrac{1}{(1 - T) \dots (1 - q^{n} T)}.
$$

\begin{proposition}
\label{SchKatz}
Given any projective scheme $X$ of dimension $n$ defined over $k$, let
$$
\tau(X) = \TotDeg \dfrac{Z(X, T)}{Z(\PP^{n}, T)} \, .
$$
Also, given any nonnegative integers $m$, $N$, and $\mathbf{d} =
(d_{1},\dots, d_{m}) \in (\NN^{\times})^{m}$, let
$$
\tau_{k}(m, N, \mathbf{d}) = \sup_{X} \ \tau(X),
$$
where the supremum is over projective schemes $X$ defined over $k$, and of
type $(m, N, \mathbf{d})$. If $\delta = \max(d_{1}, \dots, d_{m})$, then
$$
\tau_{k}(m, N, \mathbf{d}) \leq 9 \times 2^{m} \times (m \delta + 3)^{N +
1},
$$
and, in particular, $\tau_{k}(m, N, \mathbf{d})$ is bounded by a constant
independent of the field $k$.
\end{proposition}

\begin{proof}
Follows from Katz's Inequality \eqref{KatzIneq}, since $\tau (X) \leq
\sigma_{\ell}(X) + n.$
\end{proof}

\section{Number of Points of Complete Intersections}
\label{NbPtsCI}

In this section $k = \Fq$. We now state our main Theorem on the number of
points of complete intersections. The number $b'_{n}(N,\mathbf{d})$ is
defined in Theorem \ref{ICFormula}.

\begin{theorem}
\label{MainThm}
Let $X$ be an irreducible complete intersection of dimension $n$ in
$\ProjCan_{k}$, defined by $r = N - n$ equations, with multidegree
$\mathbf{d} = (d_{1}, \dots, d_{r})$, and choose an integer $s$ such that
$\dim \sing X \leq s \leq n - 1$. Then
$$
\abs{\card{X(k)} - \, \pi_{n}} \leq
b'_{n - s - 1}(N - s - 1,\mathbf{d}) \, q^{(n + s + 1)/2} +
C_{s}(X) q^{(n + s)/2},
$$
where $C_{s}(X)$ is a constant independent of $k$. If $X$ is nonsingular,
then
$C_{- 1}(X) = 0$. If $s \geq 0$, then
$$
C_{s}(X) = \sum_{i = n}^{n + s} b_{i,\ell}(\bar{X}) + \varepsilon_{i}
$$
and upon letting $\delta = \max(d_{1}, \dots, d_{r})$, we have
$$
C_{s}(X) \leq \tau(X) \leq \tau(r, N, \mathbf{d}) \leq
9 \times 2^{r} \times (r \delta + 3)^{N + 1}.
$$
\end{theorem}

\begin{proof}
Equality \eqref{FracZeta} implies
$$
\dfrac{Z(X, T)}{Z(\PP^{n}, T)} =
\dfrac{P_{1,\ell}(X, T) \cdots P_{2n - 1,\ell}(X, T) (1 - T) \dots
(1 - q^{n} T)} {P_{0,\ell}(X, T) \cdots P_{2n,\ell}(X, T)},
$$
and therefore, in view of \eqref{DefPoly00} and \eqref{GenSeries}, we get
\begin{equation}
\label{GrTF3}
\card{X(k)} - \pi_{n} = \sum_{i = 0}^{2 n} \ (- \varepsilon_{i} \,
q^{i/2} + (-1)^{i} \sum_{j = 1}^{b_{i}} \ \omega_{ij,\ell})
\end{equation}
where it may be recalled that $\varepsilon_{i} = 1$ if $i$ is even and
$0$ otherwise. Moreover, when all the cancellations have been performed, the
number of terms of the right-hand side of \eqref{GrTF3} is at most equal to
the total degree $\tau(X)$ of the rational fraction above. Now from
Proposition \ref{Bsci}\eqref{Bsci1} and \eqref{Bsci3} we have
$$
P_{i,\ell}(X, T) = 1 - \varepsilon_{i} q^{i/2} T \quad \text{if} \
i \notin [n, n + s + 1] \, .
$$
Hence, these polynomials cancel if $Z(X, T)/Z(\PP^{n}, T)$ is in
irreducible form. Accordingly, from \eqref{GrTF3} we deduce
$$
\card{X(k)} - \pi_{n} = \sum_{i = n}^{n + s + 1} \ (- \varepsilon_{i} \,
q^{i/2} + (-1)^{i} \sum_{j = 1}^{b_{i}} \ \omega_{ij,\ell}),
$$
and this gives
$$
\abs{\card{X(k)} - \pi_{n}} \leq A + B,
$$
where $A = 0 $ if $s = n - 1$ and
\begin{eqnarray*}
A & = & \abs{- \varepsilon_{n + s + 1} \,
q^{(n + s + 1)/2} + (-1)^{n + s + 1}
\sum_{j = 1}^{b_{n + s + 1}} \ \omega_{(n + s + 1)j,\ell}} \quad
\text{if} \ s < n-1,
\end{eqnarray*}
while
\begin{eqnarray*}
B & = &
\abs{\sum_{i = n}^{n + s} \ (- \varepsilon_{i} \, q^{i/2} + (-1)^{i}
\sum_{j = 1}^{b_{i}} \ \omega_{ij,\ell})} \quad
\text{if} \ s \leq n-1.
\end{eqnarray*}
Now by Deligne's Main Theorem \eqref{DeMT}, the numbers $\omega_{ij,\ell}$
are
pure of weight $\leq i/2$, and thus
$$
B \leq C_{s}(X) \, q^{(n + s)/2} \quad \text{where} \quad C_{s}(X) =
\sum_{i = n}^{n+s} b_{i,\ell}(\bar{X}) + \varepsilon_i \, .
$$
Clearly, $C_s(X)$ is independent of $k$ and is at most equal to the number
of
terms in the right-hand side of \eqref{GrTF3}. Hence by Proposition
\ref{SchKatz},
$$
C_{s}(X) \leq \tau(X) \leq \tau(r, N, \mathbf{d}) \leq
9 \times 2^{r} \times (r \delta + 3)^{N + 1}.
$$
Next, if $s = n - 1$, then $A = 0$. Suppose $s \leq n - 2$. If $n + s + 1$
is odd (in particular, if $s = n-2$), then by Proposition
\ref{Bsci}\eqref{Bsci2}.
$$
A = \abs{\sum_{j = 1}^{b_{n + s + 1}} \ \omega_{(n + s + 1)j,\ell}}
\leq b_{n - s - 1}(N - s - 1,\mathbf{d}) \, q^{(n + s + 1)/2}.
$$
On the other hand, if $0 \leq s \leq n - 3$ and if $n + s + 1$ is even,
then
$$
A =
\abs{\sum_{j = 1}^{b_{n + s + 1}} \ \omega_{(n + s + 1)j,\ell}
- q^{(n + s + 1)/2}}.
$$
But in view of Remark \ref{Even}, $H^{n + s + 1}(X,\QQ_{\ell})$ contains a
subspace which is isomorphic to $\QQ_{\ell}(-(n + s + 1)/2)$. Thus $q^{(n +
s + 1)/2}$ is an eigenvalue of the highest possible weight of $F \mid
H^{i}_{c}(\bar{X})$, and hence it is a reciprocal root of
$P_{n + s + 1}(X, T)$. It follows that
$$
A \leq (b_{n - s - 1}(N - s - 1,\mathbf{d}) - 1)\, q^{(n + s + 1)/2}.
$$
Thus in any case, $A \leq b_{n - s - 1}(N - s - 1,\mathbf{d}) \, q^{(n + s +
1)/2}$. The case when $X$ is non\-singular follows similarly using
Proposition \ref{BettiIC}.
\end{proof}

The case where $X$ is non\-singular is Deligne's Theorem \cite[Thm.
8.1]{Deligne1}. In the opposite, since we always have $s \leq n - 1$,
Theorem \ref{MainThm} implies the following weak version of the Lang-Weil
inequality for complete intersections:
$$
\card{X(k)} - \, \pi_{n} = O\left( q^{n - (1/2)} \right).
$$
We shall obtain a much better result in Theorem \ref{LangWeil}.
Nevertheless, the following Corollary shows that we can obtain the
Lang-Weil inequality (in fact, a stronger result) as soon as some mild
regularity conditions are satisfied.

If $X$ is regular in codimension one, \idest, if $\dim \sing X \leq n - 2$,
then, as stated in the beginning of Section \ref{Cohci}, $X$ is integral,
and so the hypothesis that $X$ is irreducible is automatically fulfilled.
Moreover, notice that for a complete intersection $X$ in $\PP^N$, Serre's
Criterion of Normality \cite[Thm. 5.8.6, p. 108]{EGA42} implies that $X$ is
normal if and only if it is regular in codimension one.

\begin{corollary}
\label{NormalLW}
If $X$ is a normal complete intersection of dimension $n$ in
$\ProjCan_{k}$ with multidegree $\mathbf{d}$, then
$$
\abs{\card{X(k)} - \, \pi_{n}} \leq
b'_{1}(N - n + 1,\mathbf{d}) \, q^{n - (1/2)} +
C_{n-2}(X) q^{n-1},
$$
where $C_{n-2}(X)$ is as in Theorem \ref{MainThm}. Moreover, if $d = \deg
X$,
then
$$
b'_{1}(N - n + 1,\mathbf{d}) \leq (d-1)(d-2),
$$
with equality holding if and only if $X$ is a hypersurface.
\end{corollary}

\begin{proof}
Follows from Theorem \ref{MainThm} with $s = n - 2$ and the observations in
Example \ref{BettiEx} (ii).
\end{proof}

\begin{remark}
\label{CaseR2}
It is also worthwhile to write down explicitly the particular case of a
complete intersection $X$ of dimension $n$ in $\ProjCan_{k}$ regular in
codimension 2. Namely, if $\dim \sing X \leq n - 3$, then
$$
\abs{\card{X(k)} - \, \pi_{n}} \leq
b'_{2}(N - n - 2,\mathbf{d}) \, q^{n - 1} +
C_{n - 3}(X) q^{n - (3/2)}.
$$
\end{remark}

\section{Complete Intersections with isolated singularities}
\label{CentralBettiSingular}

We now prove an inequality for the central Betti number of complete
intersections with only isolated singularities. Unfortunately, the proof of
this result depends on the following condition:

$\resol$
\emph{The resolution of singularities holds in characteristic $p$ for
excellent local rings of dimension at most $n$}.

This condition means the following. Let $A$ be an excellent local ring of
equal characteristic $p$ of dimension $\leq n$ and $X = \spec A$. Let $U$ be
a regular open subscheme of $X$, and $S = X \setminus U$. Then, there is a
commutative diagram
$$
\begin{CD}
\widetilde{U} & @>>> & \widetilde{X} & @<<< & \widetilde{S} \\
@VVV          &      & @VV{\pi}V     &      &  @VVV         \\
U             & @>>> & X             & @<<< & S
\end{CD}
$$
where $\widetilde{X}$ is a regular scheme and where $\pi$ is a proper
morphism which is a birational isomorphism, and an isomorphism when
restricted to $\widetilde{U}$. Moreover $\widetilde{S} = \widetilde{X} \setminus
\widetilde{U}$ is a \emph{divisor with normal crossings}, that is, a family
of regular schemes of pure codimension $1$ such that any subfamily
intersects properly.

Recall that the resolution of singularities holds in characteristic $0$, and
that $\MathRoman{(\mathbf{R}_{2,p})}$ holds for any $p$. For details, see
the book \cite{Abh} by Abhyankar and the survey \cite{Lipman} by Lipman.

\begin{proposition}
\label{BsciRnp}
Assume that $k = \bar{k}$ and that $\MathRoman{(\mathbf{R}_{n,p})}$ holds. Let
$X$ be a complete intersection in $\ProjCan$ of dimension $n \geq 1$ with $\dim
\sing X = 0$. Then
$$b_{n,\ell}(X) \leq b_{n}(N, \mathbf{d}).$$
\end{proposition}

\begin{proof}
Assume that $X$ is defined by the vanishing of the system of forms
$(f_{1}, \dots, f_{r})$. Since $k$ is algebraically closed, there exists a
nonsingular complete intersection $Y$ of $\ProjCan$ of codimension $r$ of
multidegree $\mathbf{d}$. Let $(g_{1}, \dots, g_{r})$ be a system defining
$Y$. The one-parameter family of polynomials
$$
T g_{j}(X_{0}, \dots, X_{N}) +
(1 - T) f_{j}(X_{0}, \dots, X_{N}), \quad (1 \leq j \leq r)
$$
define a scheme $Z$ coming with a morphism
$$
\pi_{0} : Z \longrightarrow \Aff^{1}_{k}.
$$
The morphism $\pi_{0}$ is proper and flat. Let $S = \spec k\{T\}$ where
$k\{T\}$ is the strict henselization of $k[T]$ at the ideal generated by $T$
\cite[p. 38]{Milne}. Then $S$ is the spectrum of a discrete valuation ring
with separably closed residual field, and $S$ has two points: the closed one
$s$, with residual field $k$ and the generic one $\eta$ (such a scheme is
called a \emph{trait strictement local} in French). Let $\mathcal{Z} = Z
\times_{k} S$ the scheme obtained by base change :
$$
\begin{CD}
\mathcal{Z} & @>>> & Z             \\
@V{\pi}VV   &      & @VV{\pi_{0}}V \\
S           & @>>> & \Aff_{k}^{1}
\end{CD}
$$
Then $\pi$ is a proper and flat morphism, and the closed fiber $Z_{s} =
\pi^{-1}(s)$ is isomorphic to $X$. If $\bar{\eta}$ is a geometric point of
$S$
mapping to $\eta$, the geometric fiber $Z_{\bar{\eta}} = \mathcal{Z} \times
\kappa(\bar{\eta})$ is a nonsingular complete intersection of codimension
$r$. We thus get a diagram
$$
\begin{CD}
Z_{\bar{\eta}} & @>>> & \mathcal{Z}  & @<<< & X    \\
@VVV           &      & @V{\pi}VV    &      & @VVV \\
\bar{\eta}     & @>>> & S            & @<<< & s
\end{CD}
$$
By \cite[Eq. 2.6.2, p. 9]{DeligneSGA71}, there is a long exact sequence of
cohomology
$$
\dots
\longrightarrow \phi^{i - 1}_{gl}
\longrightarrow H^{i}(Z_{s},\QQ_{\ell})
\stackrel{sp^{i}}{\longrightarrow} H^{i}(Z_{\bar{\eta}},\QQ_{\ell})
\longrightarrow \phi^{i}_{gl}
\longrightarrow \dots
$$
where $\phi^{i}_{gl}$ is the space of \emph{global vanishing cycles}, and
$sp^{i}$ is the \emph{specialization morphism}. The \emph{Theorem on Sheaves
of Vanishing Cycles} for cohomology of low degree \cite[Thm. 4.5 and Var.
4.8]{DeligneSGA71} states that $\phi^{i}_{gl} = 0$ for $i \leq n - 1$.
Hence $sp^{i}$ is an isomorphism for $i \leq n - 1$ and an injection for
$i = n$ which proves the required result, with the following warning:
according to \cite[4.4, p. 14]{DeligneSGA71}, the proof of the Theorem on
Sheaves of Vanishing Cycles relies on $\resol$.
\end{proof}

From now on let $k = \Fq$. The following corollary is essentially a result
of I.E. Shparlinski\u{\i} and A.N. Skorobogatov \cite{Sh-S}. However, as
noted in Remark \ref{ThreeRemarks}, one needs to assume $\resol$ for proving
such a result.

\begin{corollary}
\label{ShparSkoro}
Assume that $\resol$ holds. If $X$ is a complete intersection of dimension
$n$ in $\ProjCan_{k}$ with multidegree $\mathbf{d}$ with only isolated
singularities, then
$$\abs{\card{X(k)} - \, \pi_{n}} \leq
b'_{n - 1}(N - 1, \mathbf{d}) \, q^{(n + 1)/2} +
\left( b_{n}(N, \mathbf{d}) + \varepsilon_n \right) \, q^{n/2}.
$$
\end{corollary}

\begin{proof}
Use Theorem \ref{MainThm} with $s = 0$ and apply Proposition \ref{BsciRnp} to
$\bar{X}$.
\end{proof}

For hypersurfaces, this implies a worse but a particularly simple
inequality.

\begin{corollary}
Assume that $\resol$ holds. If $X$ is a hypersurface in $\PP^{n + 1}_{k}$ of
degree $d$ with only isolated singularities, then
$$
\abs{\card{X(k)} - \, \pi_{n}} \leq (d - 1)^{n + 1} \, q^{(n + 1)/2}.
$$
\end{corollary}

\begin{proof}
Follows from Corollary \ref{ShparSkoro} and the inequality in Example
\ref{BettiEx} (i).
\end{proof}

The following result may be thought of as an analogue of the Weil inequality
for certain singular curves and it is a more precise version of a result by
Aubry and Perret \cite[Cor. 2.5]{AubryPerret}.

\begin{corollary}
\label{AubryPerretIneq}
If $X$ is an irreducible curve in $\ProjCan_{k}$, then
$$
\abs{\card{X(k)} - \, (q+1) } \leq b_{1}( \bar{X}) \sqrt{q}.
$$
Moreover, if the curve $X$ is a complete intersection with multidegree
$\mathbf{d}$, and if $d$ denotes the degree of $X$, then
$$b_{1}( \bar{X}) \leq b_{1}(N , \mathbf{d}) \le (d-1)(d-2),$$
and the last inequality is an equality if and only if $X$ is isomorphic
to a nonsingular plane curve.
\end{corollary}

\begin{proof}
The first statement follows from the Trace Formula \eqref{GrTF} and
Deligne's Main Theorem \eqref{DeMT}. In the second statement, the first
inequality follows from Proposition \ref{BsciRnp} since $\resol$ is true
for $n = 1$. The last assertion is a consequence of the observations in
Example \ref{BettiEx} (ii).
\end{proof}

\begin{example}
\label{ExCone1}
We work out here a simple case which will be used later on in Example
\ref{ExCone2}. Let $C$ be a nonsingular plane curve of genus $g$ and degree
$d$, defined over $k$, and let $X$ be the projective cone in $\PP^{3}_{k}$
over $C$ \cite[Ex. 2.10, p. 13]{Ha}. Thus $X$ is a surface which is a
complete intersection, and with exactly one singular point, namely, the
vertex of the cone. Then
$$\card{X(k_{m})} = q^{m} \card{C(k_{m})} + 1.$$
If $\alpha_{1}, \dots, \alpha_{2g}$ are the roots of $F$ in $H^{1}(C)$, then
\begin{equation}
\label{NbPtsCone}
\card{X(k_{m})} =
q^{m} (q^{m} + 1 - \sum_{j = 1}^{2g} \alpha_{j}^{m}) + 1 =
q^{2m} - \sum_{j = 1}^{2g} (q \alpha_{j})^{m} + q^{m} + 1,
\end{equation}
and therefore
$$\abs{\card{X(k)} - \pi_{2}} \leq (d - 1)(d - 2) \, q^{3/2}.$$
Observe that the equality can occur if $C$ has the maximum number of
points allowed by Weil's inequality. From \eqref{NbPtsCone} and the
definition \eqref{DefZeta} of the zeta function, we get
$$
Z(X, T) = \dfrac{P_{1}(C, qT)}
{(1 - q^{2}T)(1 - qT)(1 - T)}.
$$
We compare this with the expression \eqref{FracZeta} of the zeta function.
We note that $X$ is irreducible, that the eigenvalues of the Frobenius in
$H^{3}(X)$ are pure of weight $3$, and that those of
$H^{1}(X)$ are pure of weight $\leq 1$. Hence
$$
P_{1}(X, T) = 1, \quad P_{2}(X, T) = (1 - qT), \quad
P_{3}(X, T) = P_{1}(C, qT).
$$
By looking at the degrees of these polynomials, we find
$$b_{1}(X) = 0, \quad b_{2}(X) = 1, \quad b_{3}(X) = 2 g(C).$$
\end{example}

\section{The Penultimate Betti Number}
\label{Penultimate}

Let $X$ be a separated scheme of finite type over $k = \Fq$ of dimension
$n$. Denote by $H^{2n - 1}_{+}(\bar{X})$ the subspace of $H^{2n -
1}_{c}(\bar{X})$ generated by the (generalized) eigenvectors of the
Frobenius endomorphism whose eigenvalues are pure of weight exactly equal to
$2n - 1$. This subspace is the component of maximal weight $2n - 1$ in the increasing
filtration of $H^{2n - 1}_{c}(\bar{X})$ induced by the weight. Define
$$
P_{2n - 1}^{+}(X, T) = \det(1 - T \, F \mid H^{2n - 1}_{+}(\bar{X}))
\in \ZZ_{\ell}[T].
$$
The \emph{$2n - 1$-th virtual Betti number} of Serre \cite[p. 28]{Katz2} is
$$
b^{+}_{2n - 1}(X) = \deg P_{2n - 1}^{+}(X, T) =
\dim H^{2n - 1}_{+}(\bar{X}).
$$
Recall that two schemes are \emph{birationally equivalent} if they have
isomorphic dense open sets.

\begin{proposition}
\label{BirInv}
Let $X$ be a separated scheme of dimension $n$ defined over $k$.
\begin{enumerate}
\item
\label{BirInv1}
The polynomial $P_{2n - 1}^{+}(X, t)$ has coefficients independent of
$\ell$.
\item
\label{BirInv2}
The space $H^{2n - 1}_{+}(\bar{X})$ is a birational invariant. More
precisely, if $U$ is a dense open set in $X$, then the open immersion $j : U
\longrightarrow X$ induces an isomorphism
$$
j^{*} : H^{2n - 1}_{+}(\bar{U}) \longrightarrow H^{2n - 1}_{+}(\bar{X}).
$$
\item
\label{BirInv3}
The polynomial $P_{2n - 1}^{+}(X, T)$ and the virtual Betti number
$b^{+}_{2n - 1}(X)$ are birational invariants.
\end{enumerate}
\end{proposition}

Assertion \eqref{BirInv3} is reminiscent of \cite[Cor. 6]{LangWeil}.

\begin{proof}
Assertion \eqref{BirInv1} can be easily checked by the following formula, a
particular case of \emph{Jensen's formula}. In the complex plane, let
$\gamma$ be the oriented boundary of an annulus
$$
r' \leq \card{w} \leq r, \quad \text{with} \quad q^{- n} < r' < q^{- n +
(1/2)} < r < q^{- n + 1}.
$$
If $t$ is a complex number with $\card{t} > q^{- n + 1}$, then
$$
P_{2n - 1}^{+}(X, t) = \exp \dfrac{1}{2i\pi}
\int_{\gamma} \log (1 - w^{-1}t) \ \dfrac{Z'(X, w)}{Z(X, w)} \ dw.
$$
Let us prove \eqref{BirInv2}. Set $Z = X \setminus U$ and consider the long exact
sequence of cohomology with compact support \cite[Rem. 1.30, p.
94]{Milne}:
$$
\dots \longrightarrow
H^{2n - 2}_{c}(\bar{Z}) \stackrel{i}{\longrightarrow}
H^{2n - 1}_{c}(\bar{U}) \stackrel{j_{*}}{\longrightarrow}
H^{2n - 1}_{c}(\bar{X}) \longrightarrow
H^{2n - 1}_{c}(\bar{Z}) \longrightarrow
\dots
$$
and recall that the homomorphisms of this exact sequence are
$\mathbf{g}$-equivariant. Now $H^{2n - 1}_{c}(\bar{Z}) = 0$ since $\dim Z
\leq n - 1$. Hence we get an exact sequence
$$
0 \longrightarrow
H^{2n - 2}_{c}(\bar{Z})/ \ker i \stackrel{\tilde{i}}{\longrightarrow}
H^{2n - 1}_{c}(\bar{U}) \stackrel{j_{*}}{\longrightarrow}
H^{2n - 1}_{c}(\bar{X}) \longrightarrow
0
$$
The dimension of $H^{2n - 2}_{c}(\bar{Z})$ is equal to the number of
irreducible components of $\bar{Z}$ of dimension $n - 1$, and all the
eigenvalues of the Frobenius automorphism in this space are pure of weight
$2n - 2$. Hence the eigenvalues of the Frobenius in $\Im \tilde{i}$ are also
pure of weight $2n - 2$. Thus $H^{2n - 1}_{+}(\bar{U}) \cap \Im \tilde{i} =
0$ and the restriction of $j_{*}$ to $H^{2n - 1}_{+}(\bar{U})$ is an
isomorphism. Assertion \eqref{BirInv3} is a direct consequence of
\eqref{BirInv2}.
\end{proof}

The following elementary result will be needed below.

\begin{lemma}
\label{Cauchy}
Let $A$ and $B$ be two finite sets of complex numbers included in the circle
$\card{z} = M$. If, for some $\lambda$ with $0 < \lambda < M$ and for every
integer $s$ sufficiently large,
$$
\sum_{\beta \in B} \beta^{s} - \sum_{\alpha \in A} \alpha^{s} =
O(\lambda^{s}),
$$
then $A = B$.
\end{lemma}

\begin{proof}
Suppose $A \neq B$. By interchanging $A$ and $B$ if necessary, we can
suppose that there is an element $\beta_{0} \in B \ \backslash \ A$. Then
the rational function
$$
R(z) = \sum_{\alpha \in A} \dfrac{1}{1 - \alpha z} -
\sum_{\beta \in B} \dfrac{1}{1 - \beta z}
$$
is holomorphic for $\card{z} < M^{-1}$ and admits the pole $\beta_{0}^{-1}$
on
the circle $\card{z} = M^{-1}$. But the Taylor series of
$R(z)$ at the origin is
$$
R(z) = \sum_{s = 0}^{\infty} \left(
\sum_{\alpha \in A} \alpha^{s} - \sum_{\beta \in B} \beta^{s} \right) z^{s},
$$
and the radius of convergence of this series is $ \geq \lambda^{-1} >
M^{-1}$.
\end{proof}

\begin{lemma}
\label{CritCoh}
Let $X$ be an irreducible scheme of dimension $n$ defined over $k$. The
following are equivalent :
\begin{enumerate}
\item
\label{CritCoh1}
There is a constant $C$ such that
$$
\abs{\card{X(\FF_{q^{s}})} - q^{ns}} \leq C q^{s(n - 1)}
\quad \text{for any } s \geq 1.
$$
\item
\label{CritCoh2}
We have $H^{2n - 1}_{+}(\bar{X}) = 0$.
\end{enumerate}
\end{lemma}

\begin{proof}
Let $A$ be the set of eigenvalues of the Frobenius automorphism in $H^{2n -
1}_{+}(\bar{X})$. By the Trace Formula \eqref{GrTF} and Deligne's Main
Theorem
\eqref{DeMT},
\begin{equation}
\label{RawTF}
\abs{\card{X(\mathbf{F}_{q^{s}})} - c q^{ns} -
\sum_{\alpha \in A} \alpha^{s}} \leq C'(X) q^{s(n - 1)}
\quad \text{for any } s \geq 1,
\end{equation}
where $c = \dim H^{2n}_{c}(\bar{X})$ is the number of irreducible components
of $X$ of dimension $n$ and where
$$C'(X) = \sum_{i = 0}^{2n - 2} b_{i,l}(\bar{X})$$
is independent of $q$. Here $c = 1$ since $X$ is irreducible. Now assume
that
\eqref{CritCoh1} holds. From Formula \eqref{RawTF} above, we deduce
$$
\abs{\sum_{\alpha \in A} \alpha^{s}} \leq (C'(X) + C) q^{s(n - 1)}
\quad \text{for any } s \geq 1.
$$
So Lemma \ref{Cauchy} implies $A = \emptyset$, and hence, $H^{2n -
1}_{+}(\bar{X}) = 0$. The converse implication is an immediate consequence
of
\eqref{RawTF}.
\end{proof}

\begin{lemma}
\label{ArithGen}
Let $K$ be an algebraically closed field, and $X$ an irreducible projective
curve in $\ProjCan_{K}$, with arithmetic genus $p_{a}(X)$. Let $\widetilde{X}$
be a nonsingular projective curve birationally equivalent to $X$, with
geometric genus $g(\widetilde{X})$. Then we have the following.
\begin{enumerate}
\item
\label{ArithGen1}
If $d$ denotes the degree of $X$, then
$$
2 g(\widetilde{X}) \leq b_{1}(X) \leq 2 p_{a}(X) \leq (d - 1)(d - 2).
$$
\item
\label{ArithGen2}
If $K = \bar{k}$, where $k$ is a finite field, and if $X$ is defined over $k$,
then
$$b_{1}^{+}(X) = 2 g(\widetilde{X}),$$
\end{enumerate}
\end{lemma}

During the proof of the Lemma, we shall make use of the following
standard construction, when $X$ is a curve. This leads to an inequality between Hilbert polynomials.

\begin{remark}[Comparison of Hilbert polynomials]
\label{BeqMum}
Let $K$ be an algebraically closed field, $X$ a closed subvariety
in $\ProjCan_{K}$ distinct from the whole space, and $r$ an integer such that $\dim X + 1 \leq r \leq N$. Let $\EuScript{C}_{r}(X)$ be the subvariety of $\mathcan{G}_{N - r,N}$ of linear varieties of codimension $r$ meeting $X$. From the properties of the incidence correspondence $\Sigma$ defined by
$$
\Sigma = \set{(x, E) \in \ProjCan  \times \mathcan{G}_{N - r,N}}{x \in E},
$$
it is easy to see that $\EuScript{C}_{r}(X) = \pi_{2}(\pi_{1}^{- 1}(X))$ is irreducible and that the codimension of $\EuScript{C}_{r}(X)$ in $\mathcan{G}_{N - r,N}$ is equal to $r - \dim X$.
Hence, the set of linear subvarieties of codimension $r$ in $\ProjCan_{K}$ disjoint from $X$ is a nonempty open subset $\EuScript{D}_{r}(X)$ of $\mathcan{G}_{N - r,N}$.

If $E$ belongs to $\EuScript{D}_{n + 2}(X)$, where $n = \dim X$, the
projection $\pi$ with center $E$ gives rise to a diagram
$$
\begin{CD}
X               & @>i>>  & \ProjCan_{K} - E   \\
@VV{\pi_{X}}V &        & @VV{\pi}V         \\
X'              & @>i'>> & \PP^{n + 1}_{K}
\end{CD}
$$
such that $X'$ is an irreducible hypersurface with $\deg X' = \deg X$, and
where the restriction $\pi_{X}$ is a finite birational morphism: denoting by
$S(X)$ the homogeneous coordinate ring of $X$, we have an inclusion $S(X')
\subset S(X)$, and $S(X)$ is a finitely generated module over
$S(X')$. Hence, if $P_{X}(T) \in \QQ[T]$ is the \emph{Hilbert polynomial of
$X$} \cite[p. 52]{Ha}, we have
$$
P_{X'}(t) \leq P_{X}(t) \quad \text{if} \quad t \in \NN \quad \text{and} 
\quad t \rightarrow \infty.
$$
\end{remark}

\begin{proof}[Proof of Lemma \ref{ArithGen}]
Let $U$ be a regular open subscheme of $X$. Then, there is a
commutative diagram
$$
\begin{CD}
\widetilde{U} & @>>> & \widetilde{X} & @<<< & \widetilde{S} \\
@VVV          &      & @VV{\pi}V     &      &  @VVV         \\
U             & @>>> & X             & @<<< & S
\end{CD}
$$
where $\widetilde{X}$ is a nonsingular curve, where $\pi$ is a proper
morphism which is a birational isomorphism, and an isomorphism when restricted
to $\widetilde{U}$, and
$$
\sing X \subset S = X \setminus U, \quad
\widetilde{S} = \widetilde{X} \setminus \widetilde{U}.
$$
The excision long exact sequence in compact cohomology \cite[Rem. 1.30, p.
94]{Milne} gives:
$$
\begin{CD}
0
& @>>> & H^{0}_{c}(X)
& @>>> & H^{0}_{c}(S)
& @>>> & H^{1}_{c}(U)
& @>>> & H^{1}_{c}(X)
& @>>> & 0
\end{CD}
$$
and there is another exact sequence if we replace $X, U, S$ by
$\widetilde{X},\widetilde{U},\widetilde{S}$.
This implies
$$
b_{1}(U) = b_{1}(X)  - 1 + \card{S}, \qquad
b_{1}(\widetilde{U}) = b_{1}(\widetilde{X})  - 1 + \card{\tilde{S}},
$$
and since $U$ and $\widetilde{U}$ are isomorphic, we obtain
$$
b_{1}(X) = b_{1}(\widetilde{X}) + d(X) = 2 g(\widetilde{X}) + d(X), 
\quad \text{where} \quad d(X) = \card{\tilde{S}} - \card{S},
$$
since, as is well-known, $b_{1}(\widetilde{X}) = 2g(\widetilde{X})$. Let
$$\delta(X) = p_{a}(X) - g(\widetilde{X}).$$
Then $0 \leq d(X) \leq \delta(X)$ \cite[Prop. 1, p. 68]{Serre0}. Hence
$$
b_{1}(X) = 2g(\widetilde{X}) + d(X)
\leq 2g(\widetilde{X}) + 2\delta(X) = 2 p_{a}(X).
$$
This proves the first and second inequalities of \eqref{ArithGen1}. The
Hilbert polynomial of $X$ is (\cite[p. 54]{Ha}):
$$
P_{X}(T) = d T + 1 - p_{a}(X),
$$
Apply now the construction of Remark \ref{BeqMum} to $X$, and obtain a
morphism $X \longrightarrow X'$, where $X'$ is a plane curve of degree $d$.
From the inequality $P_{X'}(t) \leq P_{X}(t)$ for $t$ large, we get $p_{a}(X)
\leq p_{a}(X')$. Now by Example \ref{BettiEx}(ii),
$$p_{a}(X') = (d - 1)(d - 2)/2,$$
since $X'$ is a plane curve of degree $d$, and so
$$p_{a}(X) \leq (d - 1)(d - 2)/2,$$
and this proves the third inequality of \eqref{ArithGen1}. Now, under the
hypotheses of \eqref{ArithGen2}, we have by \cite[Thm. 2.1]{AubryPerret}:
$$
P_{1}(X, T) = P_{1}(\widetilde{X}, T) \prod_{j = 1}^{d(X)}(1 -
\omega_{j}T),
$$
where the numbers $\omega_{j}$ are roots of unity, 
and this implies the inequality in \eqref{ArithGen2}.
\end{proof}

From now on, suppose $k = \Fq$. If $X$ is a separated scheme of finite type
over $k$, we say that the space $H^{i}_{c}(\bar{X})$ is \emph{pure} of
weight $i$ if all the eigenvalues of the Frobenius automorphism in this
space are pure of weight $i$.

\begin{proposition}
\label{PolyDiv}
Let $X$ be a closed subvariety over $k$ of dimension $n$ in $\ProjCan_{k}$
which is regular in codimension one. Then:
\begin{enumerate}
\item
\label{PolyDiv1}
The space $H^{2n - 1}(\bar{X})$ is pure of weight $2n - 1$.
\item
\label{PolyDiv2}
If $Y$ is a typical curve on $X$ over $k$, then
$P_{2n - 1}(X, T)$ divides $P_{1}(Y, q^{n - 1}T).$
\end{enumerate}
\end{proposition}

\begin{proof}
By passing to finite extension $k'$ of $k$, we can find a  
typical curve $Y$ defined over $k'$. In that case the Gysin map
$$
\iota_{*} : H^{1}(\bar{Y},\QQ_{\ell}(1 - n)) \longrightarrow
H^{2n - 1}(\bar{X}, \QQ_{\ell})
$$
is a surjection by Corollary \ref{LefschetzSGA}, and all the eigenvalues are
pure of the same weight, since $\iota_{*}$ is $\mathbf{g}$-equivariant. This
proves \eqref{PolyDiv1}, and also that $P_{2n - 1}(X, T)$ divides $P_{1}(Y,
q^{n - 1}T)$ if $k = k'$, which proves \eqref{PolyDiv2}.
\end{proof}

In view of \eqref{PolyDiv1}, a natural question is then to ask under which
conditions the space $H^{i}(\bar{X})$ is pure of weight $i$. The following
proposition summarizes the results on this topic that we can state.

\begin{proposition}
\label{PurityGen}
Let $X$ be a projective variety of dimension $n$ defined over $k$ and assume that $\dim \sing X \leq s$.
\begin{enumerate}
\item
\label{PurityGen1}
The space $H^{i}(\bar{X})$ is pure of weight $i$ if $i \geq n + s + 1$.
\end{enumerate}
Assume now that $\resol$ holds, and that $X$ is a complete intersection
with only isolated singularities. Then:
\begin{enumerate}
\setcounter{enumi}{1}
\item
\label{PurityGen2}
The space $H^{n}(\bar{X})$ is pure of weight $n$.
\end{enumerate}
\end{proposition}

\begin{proof}
By Corollary \ref{Equiv}, we can find a nonsingular proper linear section $Y$ 
of $X$ of codimension $s + 1$ defined over a finite extension $k'/k$.
Since the Gysin maps are equivariant with respect to $\Gal(\bar{k}/k')$,
Corollary \ref{LefschetzSGA} implies that  the eigenvalues of the Frobenius of
$k'$ in $H^{2n - 1}(\bar{X}, \QQ_{\ell}(n))$ are pure, and the same holds for
the eigenvalues of the Frobenius of $k$, since they are roots of the
former. This proves \eqref{PurityGen1}. Finally, as in the proof of
Proposition \ref{BsciRnp}, we deduce from the Theorem on Sheaves of
Vanishing Cycles a $\mathbf{g}$-equivariant exact sequence
$$
0 = \phi^{n - s - 1}_{gl}
\longrightarrow H^{n - s}(\bar{X})
\longrightarrow H^{n - s}(Z_{\bar{\eta}})
\longrightarrow \dots
$$
where $Z_{\bar{\eta}}$ is a nonsingular complete intersection. This implies
\eqref{PurityGen2}.
\end{proof}

\begin{remark}
If $X$ is a complete intersection with only isolated singularities, then the
spaces $H^{n}(\bar{X})$ and $H^{n + 1}(\bar{X})$ are the only ones for which
the non-primitive part is nonzero. Hence, provided $\resol$ holds,
Proposition \ref{PurityGen} shows that
\begin{enumerate}
\setcounter{enumi}{2}
\item
\label{PurityGen3}
$H^{i}(\bar{X})$ is pure of weight $i$ for $0 \leq i \leq 2n$.
\end{enumerate}
Thus for this kind of singular varieties, the situation is the same as for
nonsingular varieties. It is worthwhile recalling that if $X$ is locally the
quotient of a nonsingular variety by a finite group, then
\eqref{PurityGen3} holds without assuming $\resol$, by \cite[Rem. 3.3.11, p.
383]{Deligne2}.
\end{remark}

\section{Cohomology and Albanese Varieties}
\label{RelationAlbanese}

We begin this section with a brief outline of the construction of certain 
abelian varieties associated to a variety, namely the Albanese and Picard
varieties. Later we shall discuss their relation with some \'etale cohomology
spaces. For the general theory of abelian varieties, we refer to \cite{Lang}
and \cite{Mumford2}.

Let $X$ be a variety defined over a perfect field $k$ and assume for simplicity
that $X$ has a $k$-rational nonsingular point $x_{0}$ (by enlarging the base
field if necessary). We say that a rational map $g$ from $X$ to an abelian
variety $B$ is \emph{admissible} if $g$ is defined at $x_{0}$ and if
$g(x_{0}) = 0$. An \emph{Albanese-Weil variety} (resp. an
\emph{Albanese-Serre variety}) of $X$ is an abelian variety $A$ defined over
$k$ equipped with an admissible rational map (resp. an admissible morphism)
$f$ from $X$ to $A$ satisfying the following universal property:

$\AlbUniv$
\emph{Any admissible rational map (resp. any admissible morphism) $g$ from
$X$ to an abelian variety $B$ factors uniquely as $g = \varphi \, _{\circ} f$
for some homomorphism $\varphi : A \longrightarrow B$ of abelian varieties
defined over $k$ :}
$$
\put(25,15){\vector(1,-1){20}}
\put(35,10){\scriptsize \textit{g}}
\begin{CD}
X          &               & \\
@V{f}VV    &               & \\
A          & @>{\varphi}>> &  B
\end{CD}
$$
Assume that $A$ exists. If $U$ is an open subset of $X$ containing $x_{0}$
where $f$ is defined, then the smallest abelian subvariety containing
$f(U)$ is equal to $A$. The abelian variety $A$ is uniquely determined up to
isomorphism. Thus, the canonical map $f$ and the homomorphism $\varphi$ are
uniquely determined.

The Albanese-Serre variety $\Alb_{s} X$, together with a canonical morphism
$$f_{s} : X \longrightarrow \Alb_{s} X$$
exists for any variety $X$ \cite[Thm. 5]{Serre1}.

Let $X$ be a variety, and let $\iota : \widetilde{X} \longrightarrow X$ be
any birational morphism to $X$ from a nonsingular variety $\widetilde{X}$
(take for instance $\widetilde{X} = \reg X$). Since any rational map of a
variety into an abelian variety is defined at every nonsingular point
\cite[Thm. 2, p. 20]{Lang}, any admissible rational map of $X$ into an
abelian variety $B$ induces a morphism of $\widetilde{X}$ into $B$, and
factors through the Albanese-Serre variety $\Alb_{s} \widetilde{X}$ :
$$
\put(54,10){\vector(-1,-1){19}}
\put(45,-4){\scriptsize{$f_{w}$}}
\begin{CD}
\widetilde{X}          & @>{\iota}>>   & X       \\
@V{\tilde{f}_{s}}VV    &               & @VV{g}V \\
\Alb_{s} \widetilde{X} & @>{\varphi}>> & B
\end{CD}
$$
This implies that we can take $\Alb_{w} X = \Alb_{s} \widetilde{X}$, if we
define the canonical map as $f_{w} = \widetilde{f}_{s} \, _{\circ}
\iota^{-1}$. Hence, the Albanese-Weil variety $\Alb_{w} X$, together with a
canonical map
$$f_{w} : X \longrightarrow \Alb_{w} X$$
exists for any variety $X$, and two birationally equivalent varieties have
the same Alba\-nese-Weil variety. These two results have been proved by Weil
\cite[Thm. 11, p. 41]{Lang}, and \cite[ p. 152]{Lang}.
If $X$ is a curve, then $\Jac X = \Alb_{w} X$ by definition, and the dimension
of $\Jac X$ is equal to the genus of a nonsingular projective curve
birationally equivalent to $X$.

We recall now the following result \cite[Th. 6]{Serre1}.

\begin{proposition}[Serre]
\label{Compar}
Let $X$ be a projective variety.
\begin{enumerate}
\item
\label{Compar1}
The canonical map $f_{s} : X \longrightarrow \Alb_{s} X$ factors uniquely as
$f_{s} = \nu \, _{\circ} f_{w}$ where $\nu$ is a surjective homomorphism of
abelian varieties defined over $k$:
$$
\put(40,15){\vector(2,-1){40}}
\put(60,10){\scriptsize $f_{s}$}
\begin{CD}
X           &               & \\
@V{f_{w}}VV &               & \\
\Alb_{w} X  & @>{\nu}>>  &  \Alb_{s} X
\end{CD}
$$
\item
\label{Compar2}
If $X$ is normal, then $\ker \nu$ is connected, and $\nu$ induces an
isomorphism
$$(\Alb_{w} X)/\ker \nu \isom \Alb_{s} X.$$
\item
\label{Compar3}
if $X$ is nonsingular, then $\nu$ is an isomorphism. \hfill \qed
\end{enumerate}
\end{proposition}

\begin{example}
\label{ExCone2}
Notice that $\Alb_{s} X$ is not a birational invariant and moreover
that the inequality $\dim \Alb_{s} X < \dim \Alb_{w} X$ can occur. For
instance, let $C$ be a nonsingular plane curve of genus $g$, defined over
$k$, and let $X$ be the normal projective cone in $\PP^{3}_{k}$ over $C$, as
in Example \ref{ExCone1}. Since $X$ is an hypersurface, $\Alb_{s} X$ is
trivial by Remark \ref{PicZero} below. On the other hand, $X$ is birationally
equivalent to the nonsingular projective surface $\widetilde{X} = C \times
\PP^{1}$. Since any rational map from $\Aff^{1}$ to an abelian variety is
constant, the abelian variety $\Alb_{w} X = \Alb_{w} \widetilde{X}$ is equal
to the Jacobian $\Jac C = \Alb_{w} C$ of $C$, an abelian variety of
dimension $g$.
\end{example}

Let $X$ be a normal projective variety defined over $k$.
The \emph{Picard-Serre variety} $\Pic_{s} X$ of $X$ is the dual abelian
variety of $\Alb_{s} X$. The abelian variety $\Pic_{s} X$ should not be
confused with $\Pic_{w} X$, the \emph{Picard-Weil variety} of $X$ \cite[p.
114]{Lang}, \cite{Seshadri}, which is the dual abelian
variety of $\Alb_{w} X$ \cite[Thm. 1, p. 148]{Lang}.

One can also define the Picard-Serre variety from the \emph{Picard scheme}
$\PicS_{X/k}$ of $X$ \cite{Picard}, \cite[Ch. 8]{BLR}, which is a
separated commutative group scheme locally of finite type over $k$. Its
identity component $\PicS_{X/k}^{0}$ is an abelian scheme defined over $k$,
and $\Pic_{s} X = (\PicS^{0}_{X/k})_{\red}$ \cite[Thm. 3.3(iii), p.
237]{Picard}.

\begin{remark}
[complete intersections]
\label{PicZero}
The Zariski tangent space at the origin of $\PicS_{X/k}$ is the coherent
cohomology group $H^{1}(X, \mathcal{O}_{X})$ \cite[p. 236]{Picard}, and
hence,
$$\dim \PicS_{X/k} \leq \dim H^{1}(X, \mathcal{O}_{X}).$$
For instance, if $X$ is a projective normal complete intersection of
dimension $\geq 2$, a theorem of Serre \cite[Ex. 5.5, p. 231]{Ha} asserts
that $H^{1}(X, \mathcal{O}_{X}) = 0$; hence, $\Pic_{s} X$ and $\Alb_{s} X$
are trivial for such a scheme.
\end{remark}

If $\varphi : Y \longrightarrow X$ is a rational map defined over $k$, and
if
$f(Y)$ is reduced, then there exists one and only one homomorphism
$\varphi_{*} : \Alb_{w} Y \longrightarrow \Alb_{w} X$ defined over $k$ such
that the following diagram is commutative :
$$
\begin{CD}
Y                   & @>{\varphi}>>     & X \\
@V{\tilde{g}_{w}}VV &                   & @V{\tilde{f}_{w}}VV \\
\Alb_{w} Y          & @>{\varphi_{*}}>> & \Alb_{w} X
\end{CD}
$$
We use this construction in the following situation. Recall that the set
$\EuScript{U}_{r}(X)$ has been defined in section \ref{SingLoc}.

\begin{proposition}
\label{ChowSerreWeil}
Let $X$ be a projective variety of dimension $n$ embedded in $\ProjCan$.
If $1 \leq r \leq n - 1$, and if $E \in \EuScript{U}_{r}(X)$, let $Y =
X \cap E$ be the corresponding linear section of dimension $n - r$, and let
$\iota : Y \longrightarrow X$ be the canonical closed immersion.
\begin{enumerate}
\item
\label{CSW1}
If $n - r \geq 2$, the set of $E \in \EuScript{U}_{r}(X)$ such that
$\iota_{*}$ is a purely inseparable isogeny contains a nonempty open set of
$\mathcan{G}_{N - r,N}$.
\item
\label{CSW2}
The set of $E \in \EuScript{U}_{n - 1}(X)$ such that $\iota_{*}$ is
surjective contains a nonempty open set of $\mathcan{G}_{N - n - 1,N}$. If
$E$ belongs to this set, then $Y = X \cap E$ is a curve with
$$\dim \Alb_{w} X \leq \dim \Jac Y.$$
\item
\label{CSW3}
If $\deg X = d$, if $Y = X \cap E$ is a curve as in (ii), and 
if $\widetilde{Y}$ is a nonsingular projective curve birationally equivalent
to $Y$, then
$$
\dim \Alb_{w} X \leq g(\widetilde{Y}) \leq \dfrac{(d - 1)(d - 2)}{2}.
$$
\end{enumerate}
\end{proposition}

\begin{proof}
The results \eqref{CSW1} and \eqref{CSW2} are classical. For instance,
assertion \eqref{CSW1} follows from induction using Chow's Theorem, viz.,
Thm. 5, Ch. VIII, p.  210 in Lang's book \cite{Lang} while assertion \eqref{CSW2} 
is stated on p. 43, $\S$ 3, Ch. II in the same book. 
See also Theorem 11 and its proof in \cite[p. 159]{Serre1} for very simple arguments 
to show that $\iota_{*}$ is surjective. 
The first inequality in \eqref{CSW3} is an immediate consequences of
\eqref{CSW2} and the second is Lemma \ref{ArithGen}\eqref{ArithGen1}.
\end{proof}

\begin{remarks}
(i)
Let $X$ be a projective variety regular in codimension one, and $Y$ a
typical curve on $X$. A theorem of Weil \cite[Cor. 1 to Thm. 7]{Weil54}
states that the homomorphism
$$\iota^{*} : \Pic_{w} X \longrightarrow \Jac Y$$
induced by $\iota$ has a finite kernel, which implies \eqref{CSW2} by
duality in this case.
\\ (ii)
Up to isogeny, any abelian variety $A$ appears as the Albanese-Weil
variety of a surface. To see this, it suffices to take a suitable linear section.
\end{remarks}

Let $A$ be an abelian variety defined over $k$, of dimension $g$. For each
integer $m \geq 1$, let $A_{m}$ denote the group of elements $a \in
A(\overline{k})$ such that $ma = 0$. Let $\ell$ be a prime number different
from the characteristic of $k$. The \emph{$\ell$-adic Tate module}
$T_{\ell}(A)$ of $A$ is the projective limit of the groups $A_{\ell^{n}}$,
with respect to the maps induces by multiplication by $\ell$; this is a free
$\ZZ_{\ell}$-module of rank $2g$, and the group $\mathbf{g}$ operates
on $T_{\ell}(A)$. The tensor product
$$V_{\ell}(A) = T_{\ell}(A) \otimes_{\ZZ_{\ell}} \QQ_{\ell}$$
is a vector space of dimension $2g$ over $\QQ_{\ell}$.

We recall the following result \cite[Lem. 5]{KatzLang} and \cite[Cor. 4.19,
p. 131]{Milne}, which gives a \emph{description in purely algebraic terms}
of $H^{1}(\bar{X}, \QQ_{\ell})$ when $X$ is a normal projective variety.

\begin{proposition}
\label{PicToH}
Let $X$ be a normal projective variety defined over $k$. Then, there is a
$\mathbf{g}$-equivariant isomorphism
$$
h_{X} : V_{\ell}(\Pic_{s} X)(-1) \isom H^{1}(\bar{X}, \QQ_{\ell}).
$$
In particular, $b_{1, \ell}(\bar{X}) = 2 \dim \Pic_{s} X$ is independent
of $\ell$.
\hfill \qedbox
\end{proposition}

\begin{remarks}
(i)
If $X$ is a normal complete intersection, then Proposition \ref{PicToH} and
Example \ref{PicZero} imply $H^{1}(\bar{X}, \QQ_{\ell}) = 0$, in accordance
with Proposition \ref{Bsci}\eqref{Bsci3}. \\
(ii)
If $X$ is a normal projective variety, we get from Proposition \ref{PicToH}
a $\mathbf{g}$-equivariant isomorphism
$$H^{1}(\Alb_{s} \bar{X}, \QQ_{\ell}) \isom H^{1}(\bar{X},\QQ_{\ell}).$$
\end{remarks}

\begin{proposition}
\label{AlbanToHNormal}
Let $X$ be a normal projective variety of dimension $n \geq 2$ defined
over $k$ which is regular in codimension $2$. Then there is a
$\mathbf{g}$-equivariant isomorphism
$$
j_{X} : V_{\ell}(\Alb_{w} X) \isom H^{2n - 1}(\bar{X},\QQ_{\ell}(n)).
$$
If $\resol$ holds, the same conclusion is true if one only assumes that $X$ is regular in codimension $1$.
\end{proposition}

\begin{proof}
\emph{Step $1$}.
Assume that $X$ is a subvariety in $\mathcan{P}^{N}_{k}$. Since $X$ is regular
in codimension $2$, we deduce from Proposition \ref{BertiniFlag} and Corollary
\ref{Equiv} that $\EuScript{U}_{n - 2}(X)$ contains a nonempty Zariski
open set $U_{0}$ in the Grassmannian $\mathcan{G}_{N - n + 2,N}$.
On the other hand, any open set defined over $\bar{k}$
is defined over a finite extension $k'$, and contains an open set defined over $k$
(take the intersection of the transforms by the Galois group of $k'/k$).
Let $U_{1} \subset U_{0}$ be an open set defined over $k$.
If $E \in U_{1}$, then $Y = X \cap E$ is a \emph{typical surface} on $X$ over $k$, \idest, a
nonsingular proper linear section of dimension $2$ in $X$. For such a typical
surface $Y$, the closed immersion $\iota : Y \longrightarrow X$ induces a
homomorphism $\iota_{*} : \Alb_{w} Y \longrightarrow \Alb_{w} X$.
By Proposition \ref{ChowSerreWeil}\eqref{CSW1}, the set $U$ of linear
varieties $E \in U_{1}$ such that $\iota_{*}$ is a purely inseparable isogeny
contains as well a nonempty open subset $U \subset \mathcan{G}_{N - n + 2,N}$ which is
defined over $k$.

\emph{Step $2$}.
Assume that $U(k)$ is nonempty. If $E \in U(k)$, we get a
$\mathbf{g}$-equivariant isomorphism
$$V_{\ell}(\iota_{*}) : V_{\ell}(\Alb_{w} Y) \isom V_{\ell}(\Alb_{w} X).$$
Since $Y$ is nonsingular, we get from Poincar\'e Duality Theorem for
nonsingular varieties \cite[Cor. 11.2, p. 276]{Milne} a
$\mathbf{g}$-equivariant nondegenerate pairing
\begin{equation*}
H^{1}(\bar{Y},\QQ_{\ell}) \times H^{3}(\bar{Y},\QQ_{\ell}(2))
\longrightarrow \QQ_{\ell},
\end{equation*}
from which we deduce a $\mathbf{g}$-equivariant isomorphism
$$
\psi : \Hom(H^{1}(\bar{Y}, \QQ_{\ell}),\QQ_{\ell}) \longrightarrow
H^{3}(\bar{Y}, \QQ_{\ell}(2)).
$$
Since $(X, Y)$ is a semi-regular pair with $Y$ nonsingular, from Corollary
\ref{LefschetzSGA} we know that the Gysin map
$$
\iota_{*} : H^{3}(\bar{Y},\QQ_{\ell}(2 - n)) \longrightarrow
H^{2n - 1}(\bar{X},\QQ_{\ell})
$$
is an isomorphism. Now a $\mathbf{g}$-equivariant isomorphism of vector spaces
over $\QQ_{\ell}$:
$$
j_{X} : V_{\ell}(\Alb_{w} X) \isom H^{2n - 1}(\bar{X},\QQ_{\ell}(n))
$$
is defined as the isomorphism making the following diagram commutative:
$$
\begin{CD}
\Hom(V_{\ell}(\Pic_{s} Y)(-1),\QQ_{\ell}) & @>{\varpi}>{\sim}>
& V_{\ell}(\Alb_{w} Y) & @>{V_{\ell}(\iota_{*})}>{\sim}>
& V_{\ell}(\Alb_{w} X) \\
@V{^{t}h_{Y}}V{\sim}V & & & & & & @V{j_{X}}VV \\
\Hom(H^{1}(\bar{Y}, \QQ_{\ell}),\QQ_{\ell}) & @>{\psi}>{\sim}>
& H^{3}(\bar{Y}, \QQ_{\ell})(2) & @>{\iota_{*}}>{\sim}>
& H^{2n - 1}(\bar{X}, \QQ_{\ell}(n))
\end{CD}
$$
Here $\varpi$ is defined by the Weil pairing, and $^{t}h_{Y}$ is the
transpose of the map $h_{Y}$ defined in Proposition \ref{PicToH}. Hence, the
conclusion holds if $U(k) \neq \emptyset$.

\emph{Step $3$}.
Assume that $k$ is an infinite field. One checks successively that if $U$ is an
open subset in an affine line, an affine space, or a Grassmannian, then $U(k)
\neq \emptyset$ and the conclusion follows from Step
$2$.

\emph{Step $4$}. Assume that $k$ is a finite field. Then the
following elementary result holds (as a consequence of
Proposition \ref{BoundAlgSet} below, for instance).

\begin{claim*}
\label{FiniteOpen}
Let $U$ be a nonempty Zariski open set in $\mathcan{G}_{r,N}$, defined over
$k$, and $k_{s} = \FF_{q^{s}}$ the extension of degree $s$ of $k = \Fq$. Then
there is an integer $s_{0}(U)$ such that $U(k_{s}) \neq \emptyset$ for every
$s \geq s_{0}(U)$.
\end{claim*}

Now take for $U$ the open set in $\mathcan{G}_{N - n + 2,N}$ introduced in
Step $1$. Choose any $s \geq s_{0}(U)$, and let $\mathbf{g}_{s}
= \Gal(\bar{k} / k_{s})$. Since
$U(k_{s}) \neq \emptyset$, upon replacing $k$ by $k_{s}$, we deduce from Step $2$
a $\mathbf{g}_{s}$-equivariant isomorphism of $\QQ_{\ell}$-vector spaces:
$$
j_{X, s} : V_{\ell}(\Alb_{w} X) \isom H^{2n - 1}(\bar{X},\QQ_{\ell}(n)).
$$
This implies in particular that if $m = 2 \dim \Alb_{w} X$, then
$$
\dim H^{2n - 1}(\bar{X},\QQ_{\ell}(n)) = \dim V_{\ell}(\Alb_{w} X) = m.
$$
In each of these spaces, there is an action of $\mathbf{g} = \mathbf{g}_{1}$.
By choosing bases, we identify both of them with $\QQ_{\ell}^{m}$. Denote by
$g_{1} \in \GL_{m}(\QQ_{\ell})$ the matrix of the endomorphism
$V_{\ell}(\varphi)$, where $\varphi \in \mathbf{g}$ is the geometric
Frobenius, and by $g_{2} \in \GL_{m}(\QQ_{\ell})$ the matrix of the
Frobenius operator in  $H^{2n - 1}(\bar{X},\QQ_{\ell}(n))$. The existence of
the $\mathbf{g}_{s}$-equivariant isomorphism $j_{X, s}$ implies that
$g_{1}^{s}$ and $g_{2}^{s}$ are conjugate. In order to finish the proof when
$k$ is finite, we must show that $g_{1}$ and $g_{2}$ are conjugate.
This follows from the Conjugation Lemma below, since  $g_{1}$ is
semi-simple by \cite[p. 203]{Mumford2}.

\emph{Step $5$}. Assume now that $\resol$ holds and that $X$ is regular in codimension $1$. 
Take $\widetilde{X}$ to be a nonsingular projective variety birationally equivalent to $X$ over $k$.
Then $\Alb_{w} \widetilde{X} = \Alb_{w} X$ since the Albanese-Weil variety is a
birational invariant, and
$$
H^{2n - 1}_{+}(\widetilde{X} \otimes \bar{k}, \QQ_{\ell}(n)) =
H^{2n - 1}_{+}(X \otimes \bar{k}, \QQ_{\ell}(n)),
$$
by Proposition 8.1(ii). Now it is well known that
$H^{2n - 1}(\widetilde{X} \otimes \bar{k}, \QQ_{\ell})$ is pure, and the
same holds for $X$, by Prop. 8.7(i). Hence,
$$
H^{2n - 1}(\widetilde{X} \otimes \bar{k}, \QQ_{\ell}(n)) =
H^{2n - 1}(X \otimes \bar{k}, \QQ_{\ell}(n)).
$$
Since the conclusion is true for a \emph{nonsingular} variety, we obtain a
$\mathbf{g}$-equivariant map
$$
j_{\widetilde{X}} : V_{\ell}(\Alb_{w} \widetilde{X}) \isom H^{2n -
1}(\widetilde{X} \otimes \bar{k},\QQ_{\ell}(n)),
$$
and this gives the required $\mathbf{g}$-equivariant isomorphism.
\end{proof}

\begin{ConjLemma*}
\label{LemConjug}
Let $K$ be a field of characteristic zero, and let $g_{1}$ and $g_{2}$ be two matrices
in $\GL_{n}(K)$, with $g_{1}$ semi-simple. If $g_{2}^{s}$ is conjugate to $g_{1}^{s}$ for infinitely many prime numbers $s$, then $g_{2}$ is conjugate to $g_{1}$.
\end{ConjLemma*}

\begin{proof}
Let $g_{2} = s u$ be the multiplicative Jordan decomposition of $g_{2}$ into
its semi-simple and unipotent part. Take $a$ and $b$ prime with $g_{2}^{a}$
conjugate to $g_{1}^{a}$. Then $s^{a} u^{a}$ is conjugate to
$g_{1}^{a}$, and hence, $u^{a} = \mathbf{I}$, by the uniqueness of the Jordan
decomposition. Similarly, we find $u^{b} = \mathbf{I}$. Hence $u =
\mathbf{I}$ with the help of B\'ezout's equation, and $g_{2}$ is semisimple.

Take now two diagonal matrices $d_{1}$ and $d_{2}$ in $\GL_{n}(\bar{K})$ such
that $g_{i}$ is conjugate to $d_{i}$ in $\GL_{n}(\bar{K})$. Two conjugate
diagonal matrices are conjugate by an element of the group $W$ of permutation
matrices: if $d_{1}^{s}$ and $d_{2}^{s}$ are conjugate, then $d_{2}^{s}
= (w_{s} d_{1} w_{s}^{- 1})^{s}$ with
$w_{s} \in W$. Since $W$ is finite, one of the sets
$$
T(w) = \set{s \in \NN}{d_{2}^{s} = (w d_{1} w^{- 1})^{s}}
$$
contains infinitely many prime numbers. Take two prime numbers $a$
and $b$ in that set, then
$$
d_{2}^{a} = h_{1}^{a}, \quad d_{2}^{b} = h_{1}^{b},
\quad h_{1} = w d_{1} w^{- 1},
$$
from which we deduce $d_{2} = h_{1}$ by B\'ezout's equation. This implies
that $d_{1}$ and $d_{2}$  are conjugate in $\GL_{n}(\bar{K})$, and the same
holds for $g_{1}$ and $g_{2}$. But two elements of
$\GL_{n}(K)$ which are conjugate in $\GL_{n}(\bar{K})$ are conjugate in
$\GL_{n}(K)$.
\end{proof}

\begin{remark}
\label{TrivAlb}
Let $X$ be a complete intersection of dimension $\geq 2$ which is regular in
codimension $2$. Then Proposition \ref{AlbanToHNormal} implies that the
Albanese-Weil variety $\Alb_{w} X$ is trivial, since Proposition
\ref{Bsci}\eqref{Bsci1} implies that $H^{2n - 1}(\bar{X},\QQ_{\ell}) = 0$.
\end{remark}

The following result is a weak form of Poincar\'e Duality between the first
and the penultimate cohomology spaces of some singular varieties.

\begin{corollary}
\label{InjH}
Let $X$ be a normal projective variety of dimension $n \geq 2$ defined over
$k$ regular in codimension $2$. Then there is a $\mathbf{g}$-equivariant
injective linear map
$$
H^{1}(\bar{X},\QQ_{\ell}) \longrightarrow
\Hom(H^{2n - 1}(\bar{X},\QQ_{\ell}(n)),\QQ_{\ell}).
$$
If $\resol$ holds, the same conclusion is true if one only assumes that $X$ is regular in codimension $1$.
\end{corollary}

\begin{proof}
Proposition \ref{PicToH} furnishes an isomorphism
$$
h_{X}^{-1} : H^{1}(\bar{X}, \QQ_{\ell}) \isom V_{\ell}(\Pic_{s} X)(-1).
$$
From the surjective map
$\nu : \Alb_{w} X \longrightarrow \Alb_{s} X$
defined in Proposition \ref{Compar}, we get by duality a homomorphism with
finite kernel
${^{t}}\nu : \Pic_{s} X \longrightarrow \Pic_{w} X$
generating an injective homomorphism
$$
V_{\ell}({^{t}}\nu) :
V_{\ell}(\Pic_{s} X) \longrightarrow V_{\ell}(\Pic_{w} X).
$$
Now the Weil pairing induces an isomorphism
$$
V_{\ell}(\Pic_{w} X)(-1) \longrightarrow
\Hom(V_{\ell}(\Alb_{w} X), \QQ_{\ell})
$$
and Proposition \ref{AlbanToHNormal} gives an isomorphism
$$
\Hom(V_{\ell}(\Alb_{w} X), \QQ_{\ell}) \longrightarrow
\Hom(H^{2n - 1}(\bar{X},\QQ_{\ell}(n)),\QQ_{\ell}).
$$
The required linear map is the combination of all the preceding maps.
\end{proof}

Recall that a projective variety, regular in codimension one, which is a
local complete intersection, is normal. For such varieties, we obtain a
sharper version of Proposition \ref{Compar}\eqref{Compar2} as follows.

\begin{corollary}
\label{ResIsog}
Assume that $\resol$ holds. Let $X$ be a projective variety, regular in
codimension one, which is a local complete intersection. Then the canonical map
$$\nu : \Alb_{w} X  \longrightarrow  \Alb_{s} X$$
is an isomorphism.
\end{corollary}

\begin{proof}
Since $\resol$ holds and since $X$ is a local complete intersection, we know,
by Poincar\'e Duality (Remark
\ref{PoincDual}), that
$$
\dim H^{1}(\bar{X},\QQ_{\ell}) = \dim H^{2n - 1}(\bar{X},\QQ_{\ell}(n)),
$$
and the linear map of Corollary \ref{InjH} is bijective. Hence, the
homomorphism
$$
V_{\ell}({^{t}}\!\nu) :
V_{\ell}(\Pic_{s} X) \longrightarrow V_{\ell}(\Pic_{w} X)
$$
is also bijective. By Tate's Theorem \cite[Appendix I]{Mumford2}, this implies that $\nu$ is an isogeny.
Since the kernel of $\nu$ is connected by Prop. \ref{Compar}\eqref{Compar2}, it is trivial.
\end{proof}

\section{A Conjecture of Lang and Weil}
\label{LangWeilConj}

We assume now that $k = \Fq$ is a finite field. In order to state the
results of this section, we introduce a weak form of the resolution of
singularities for a variety $X$ of dimension $\geq 2$, which is of course
implied by $\resol$.

\medskip \noindent
$\resolTwo$
\emph{$X$ is birationally equivalent to a normal projective variety
$\widetilde{X}$ defined over $k$, which is regular in codimension $2$.}
\medskip

As a special case of Abhyankar's results \cite{Abh}, $\resolTwo$ is valid
in any characteristic $p > 0$ if $\dim X \leq 3$ (except perhaps when
$\dim X = 3$ and $p = 2, 3, 5$).

The following gives a \emph{description in purely algebraic terms} of the
birational invariant $H^{2n - 1}_{+}(\bar{X},\QQ_{\ell})$. Recall that
$H^{2n - 1}(\bar{X},\QQ_{\ell})$ is in fact pure as soon as $X$ is regular in codimension $1$.
\begin{theorem}
\label{AlbanToH}
Let $X$ be a variety of dimension $n \geq 2$ defined over $k$ satisfying
$\resolTwo$. Then there is a $\mathbf{g}$-equivariant isomorphism
$$
j_{X} : V_{\ell}(\Alb_{w} X) \isom H^{2n - 1}_{+}(\bar{X},\QQ_{\ell}(n)),
$$
and hence $b_{2n - 1}^{+}(X) = 2 \dim \Alb_{w} X$. In particular this
number is even.
\end{theorem}

\begin{proof}
Take $\widetilde{X}$ birationally equivalent to $X$ as in $\resolTwo$.
Then $\Alb_{w} \widetilde{X} = \Alb_{w} X$ since the Albanese-Weil
variety is a birational invariant, and
$$
H^{2n - 1}(\widetilde{X} \otimes \bar{k}, \QQ_{\ell}) =
H^{2n - 1}_{+}(X \otimes \bar{k}, \QQ_{\ell})
$$
(equality as $\mathbf{g}$-modules) by Propositions \ref{BirInv}\eqref{BirInv2} and
\ref{PurityGen}\eqref{PurityGen1}.
Now apply Proposition \ref{AlbanToHNormal} to $\widetilde{X}$.
\end{proof}

\begin{remark}
\label{Motives}
The preceding result can be interpreted in the category of \emph{motives}
over $k$ (cf. \cite{Serre4}, \cite{Motives}): this is what we mean by a
\emph{``description in purely algebraic terms''} of a cohomological
property. Let us denote by $h(X)$ the motive of a variety $X$ of dimension
$n$ defined over $k$, and by $h^{i}(X)$ the (mixed) component of $h(X)$
corresponding to cohomology of degree $i$. One can identify the category of
abelian varieties up to isogenies with the category of pure motives of
weight $- 1$. Denote by $\LL = h^{2}(\PP^{1})$ be the \emph{Lefschetz
motive}, which is pure of weight $2$. Proposition \ref{PicToH} implies that
if $X$ is normal, then
$$h^{1}(X) = \Pic_{s}(X) \otimes \LL.$$
Now, denote by $h^{i}_{+}(X)$ the part of $h^{i}(X)$ which is pure of
weight $i$. Theorem \ref{AlbanToH} means that if $X$ satisfies $\resolTwo$,
then
$$h^{2 n - 1}_{+}(X) = \Alb_{w}(X) \otimes \LL^{n}.$$
These results are classical if $X$ is nonsingular.
\end{remark}

\begin{remark}
\label{AbelianPiece}
Similarly, Theorem \ref{AlbanToH} is in accordance with the following
conjectural statement of Grothendieck \cite[p. 343]{Groth0} in the case
$i = n$ :

\noindent
\textit{`` In odd dimensions, the piece of maximal filtration of
$H^{2i-1}(X, \ZZ_{\ell}(i))$ is also the greatest ``abelian piece'', and
corresponds to the Tate module of the intermediate Jacobian $J^i(X)$
(defined by the cycles algebraically equivalent to $0$ of codimension $i$ on
$X$). ''}

\noindent
Notice that the group of cycles algebraically equivalent to $0$ of
codimension $n$ on $X$ maps in a natural way to the Albanese variety of
$X$.
\end{remark}

\begin{remark}
\label{RemATH}
In the first part of their Proposition in \cite[p. 333]{BS}, Bombieri and
Sperber state Theorem \ref{AlbanToH} without any assumption about resolution
of singularities, but they give a proof only if $\dim X \leq 2$, in which
case $\resolTwo$ holds in any characteristic, by Abhyankar's results.
\end{remark}

For any separated scheme $X$ of finite type over $k$, the group $\mathbf{g}$
operates on $\bar{X}$ and the \emph{Frobenius morphism} of $\bar{X}$ is the
automorphism corresponding to the geometric element $\varphi \in
\mathbf{g}$. If $A$ is an abelian variety defined over $k$, then $\varphi$
is an endomorphism of $A$. It induces an endomorphism
$T_{\ell}(\varphi)$ of $T_{\ell}(A)$, and
\begin{equation*}
\label{PolCarFrob}
\deg(n.1_{A} - \varphi) = f_{c}(A, n),
\end{equation*}
where
$$f_{c}(A, T) = \det(T - T_{\ell}(\varphi)).$$
The polynomial $f_{c}(A, T)$ is a monic polynomial of degree $2g$ with
coefficients in $\ZZ$, called the \emph{characteristic polynomial} of
$A$. Moreover, if $\dim A = g$,
$$f_{c}(A, T) = \prod_{j = 1}^{2g} (T - \alpha_{j}),$$
where the \emph{characteristic roots} $\alpha_{j}$ are pure of weight one
\cite[p. 139]{Lang}, \cite[p. 203-206]{Mumford2}. The constant term of
$f_{c}(A, T)$ is equal to
\begin{equation}
\label{ProdRoots}
\deg(\varphi) = \det T_{\ell}(\varphi) = \prod_{j = 1}^{2g}\alpha_{j} =
q^{g}.
\end{equation}
The \emph{trace} of $\varphi$ is the unique rational integer $\Tr(\varphi)$
such that
$$
f_{c}(A, T) \equiv
T^{2g} - \Tr(\varphi) \, T^{2g - 1} \, (\mod T^{2g - 2}).
$$
In order to state the next results, we need to introduce some
conventions. If $X$ is a separated scheme of finite type over $k$, we call
$\indep$ the following set of conditions about $H^{i}_{c}(\bar{X},
\QQ_{\ell})$:
\begin{itemize}
{\itshape
\item
The action of the Frobenius morphism $F$ in $H^{i}_{c}(\bar{X},
\QQ_{\ell})$ is diagonalizable.
\item
The space $H^{i}_{c}(\bar{X}, \QQ_{\ell})$ is pure of weight $i$.
\item
The polynomial $P_{i}(X, T)$ has coefficients in $\ZZ$ which are
independent of $\ell$.}
\end{itemize}
Furthermore, for any polynomial $f$ of degree $d$, we write its
reciprocal
polynomial as $f^{\vee}(T) = T^{d} f(T^{-1})$.

\begin{corollary}
\label{PolPic}
Let $X$ be a normal projective variety defined over $k$. Then:
\begin{enumerate}
\item
\label{PolPic1}
We have
$$P_{1}(X, T) = f_{c}^{\vee}(\Pic_{s} X, T).$$
\item
\label{PolPic3}
If $\varphi$ is the Frobenius endomorphism of $\Pic_{s} X$, then
$$\Tr(F \mid H^{1}(\bar{X}, \QQ_{\ell})) = \Tr(\varphi).$$
\item
\label{PolPic4}
Conditions $\indep$ hold for $H^{1}(\bar{X}, \QQ_{\ell})$.
\end{enumerate}
\end{corollary}

\begin{proof}
Let $A$ be an abelian variety of dimension $g$ defined over $k$ with
$$f_{c}(A, T) = \prod_{j = 1}^{2g} (T - \alpha_{j}).$$
Since the arithmetic Frobenius $F$ is the inverse of $\varphi$,
$$
\det(T - V_{\ell}(F) \mid V_{\ell}(A)) =
\prod_{j = 1}^{2g} (T - \alpha_{j}^{-1}).
$$
Since the map $\alpha \mapsto q \alpha^{-1}$ is a permutation of the
characteristic roots, we have
$$
\det(T - V_{\ell}(F) \mid V_{\ell}(A)(-1)) =
\prod_{j = 1}^{2g} (T - q \alpha_{j}^{-1}) = f_{c}(A, T),
$$
which implies
$$
\det(1 - T \ V_{\ell}(F)\mid V_{\ell}(A)(-1)) = f_{c}^{\vee}(A, T).
$$
If we apply the preceding equality to $A = \Pic_{s} X$, we get
\eqref{PolPic1} with the help of Proposition \ref{PicToH}. Now
\eqref{PolPic3} is an immediate consequence of \eqref{PolPic1} by looking at
the coefficient of $T^{2g - 1}$. As stated above, the polynomial $f_{c}(A,
T)$ belongs to $\ZZ[T]$, and its roots are pure of weight $1$. By
\cite[Prop., p. 203]{Mumford2}, the automorphism $V_{\ell}(\varphi)$ of
$V_{\ell}(A)$ is diagonalizable, and hence, \eqref{PolPic4} follows again
from Proposition \ref{PicToH}.
\end{proof}

\begin{corollary}
If $X$ is a normal projective surface defined over $k$, then the
polynomials $P_{i}(X, T)$ are independent of $\ell$ for $0 \leq i \leq 4$.
\end{corollary}

\begin{proof}
The polynomial $P_{3}(X, T) = P_{3}^{+}(X, T)$ is independent of $\ell$ by
Proposition \ref{BirInv}\eqref{BirInv1}. Moreover $P_{1}(X, T)$ is
independent of $\ell$ by Corollary \ref{PolPic}. Hence we have proved that
these polynomials are independent of $\ell$ for all but one value of $i$,
namely $i = 2$. Following an observation of Katz, the last one must also be
independent of $\ell$, since
$$
Z(X, T) =
\dfrac{P_{1}(X,T) P_{3}(X,T)}{P_{0}(X,T) P_{2}(X,T) P_{4}(X,T)} \, ,
$$
and the zeta function $Z(X, T)$ is independent of $\ell$.
\end{proof}

The following result gives an explicit description of the birational
invariants
$$P^{+}_{2n - 1}(X, T), \quad b_{2n - 1}^{+}(X), \quad
\Tr(F \mid H^{2n - 1}_{+}(\bar{X}, \QQ_{\ell}))
$$
in purely algebraic terms.

\begin{theorem}
\label{GenConj}
Let $X$ be a variety of dimension $n \geq 2$ defined over $k$.
\begin{enumerate}
\item
\label{CorrConj1}
If $g = \dim \Alb_{w} X$, then
$$P^{+}_{2n - 1}(X, T) = q^{- g}f_{c}(\Alb_{w} X, q^{n}T).$$
In particular, $b_{2n - 1}^{+}(X) = 2 g$.
\item
\label{CorrConj2}
If $\varphi$ is the Frobenius endomorphism of $\Alb_{w} X$, then
$$\Tr(F \mid H^{2n - 1}_{+}(\bar{X}, \QQ_{\ell})) = q^{n - 1}
\Tr(\varphi).$$
\end{enumerate}
\end{theorem}

\begin{remark}
\label{CorrectAndProve}
Serge Lang and Andr\'e Weil have conjectured \cite[p. 826-827]{LangWeil}
that the equality
$$P^{+}_{2n - 1}(X, T) = q^{-g}f_{c}(\Pic_{w} X, q^{n}T)$$
holds if $X$ is a variety defined over $k$, provided $X$ is complete and
nonsingular. If $X$ is only assumed to be normal, then $\Pic_{w}
X$ and $\Alb_{w} X$ are isogenous, and hence, $f_{c}(\Alb_{w} X, T) =
f_{c}(\Pic_{w} X, T)$. Thus the Lang-Weil Conjecture is a particular case
of Theorem \ref{GenConj}\eqref{CorrConj1}.

Example \ref{ExCone2} shows that we cannot replace $\Pic_{w} X$ by
$\Pic_{s} X$ in the statement of Theorem \ref{GenConj}, even if $X$ is
projectively normal.
\end{remark}

The full proof of this theorem will be given in section \ref{LWIneq}. We
first prove:

\begin{proposition}
\label{CorrConj}
If $X$ satisfies $\resolTwo$, then the conclusions of Theorem
\ref{GenConj} hold true. Moreover
\begin{enumerate}
\setcounter{enumi}{2}
\item
\label{CorrConj3}
Conditions $\indep$ hold for $H^{2n - 1}_{+}(\bar{X}, \QQ_{\ell})$.
\end{enumerate}
\end{proposition}

\begin{proof}
Let $A$ be an abelian variety of dimension $g$ defined over $k$. Then
\begin{equation}
\label{CorrConj4}
f_{c}(A, T) = \prod_{j = 1}^{2g} (T - \alpha_{j})
= q^{g} \prod_{j = 1}^{2g} (1 - \alpha_{j}^{-1} T),
\end{equation}
the last equality coming from \eqref{ProdRoots}. Now
$$
\det(T - V_{\ell}(F) \mid V_{\ell}(A)(- n)) =
\prod_{j = 1}^{2g} (T - q^{n}\alpha_{j}^{-1}).
$$
Hence
\begin{equation}
\label{CorrConj5}
\det(1 - T \ V_{\ell}(F) \mid V_{\ell}(A)(- n)) =
\prod_{j = 1}^{2g} (1 - q^{n}\alpha_{j}^{-1} T),
\end{equation}
and we get from \eqref{CorrConj4}
$$
\det(1 - T \ V_{\ell}(F) \mid V_{\ell}(A)(- n)) = q^{- g} f_{c}(A,
q^{n}T).
$$
Hence, \eqref{CorrConj1} follows from Proposition \ref{AlbanToH}. From
\eqref{CorrConj5} we deduce that $\beta$ is an eigenvalue of $F$ in
$H^{2n - 1}(\bar{X}, \QQ_{\ell})$ if and only if $q^{n}/\beta$ is among the
characteristic roots of $\varphi$. Now since $\alpha_{j}
\overline{\alpha}_{j} = q$, we get:
$$
\Tr(F \mid H^{2n - 1}(\bar{X}, \QQ_{\ell}))
= \sum_{j = 0}^{2g} \dfrac{q^n}{\alpha_{j}}
= q^{n - 1} \sum_{j = 0}^{2g} \overline{\alpha}_{j}
= q^{n - 1} \Tr(\varphi).
$$
Finally \eqref{CorrConj3} follows from Theorem \ref{AlbanToH} as in the
proof of Corollary \ref{PolPic}.
\end{proof}

We end this section by stating a weak form of the functional equation
relating the polynomials $P_{1}(X, T)$ and $P^{+}_{2n - 1}(X, T)$ when
$X$ is nonsingular.

\begin{corollary}
\label{FuncEq}
Let $X$ be a normal projective variety of dimension $n \geq 2$, defined
over $k$. If $g = \dim \Alb_{w} X$, then
$$
q^{-g} P_{1}^{\vee}(X, q^{n}T)
\quad \text{divides} \quad
P_{2n - 1}^{+}(X, T).
$$
\end{corollary}

\begin{proof}
By Corollary \ref{PolPic}\eqref{PolPic1} and Theorem
\ref{GenConj}\eqref{CorrConj1}, we know that
$$P_{1}^{\vee}(X, T) = f_{c}(\Pic_{s} X, T), \quad
P^{+}_{2n - 1}(X, T) = q^{- g} f_{c}(\Alb_{w} X, q^{n}T),$$
and Proposition \ref{Compar}\eqref{Compar1} implies that $f_{c}(\Pic_{s} X,
T)$ divides $f_{c}(\Alb_{w} X, T)$.
\end{proof}

With Corollary \ref{ResIsog}, we prove in the same way :

\begin{corollary}
Assume that $\resol$ holds. Let $X$ be a projective variety of dimension $n$ defined over $k$, regular in codimension one, which is a local complete intersection. 
If $g = \dim \Alb_{w} X$, then
$$
q^{-g} P_{1}^{\vee}(X, q^{n}T) = P_{2n - 1}(X, T).
\rlap \qedbox
$$
\end{corollary}

\section{On the Lang-Weil Inequality}
\label{LWIneq}

In this section, let $k = \FF_{q}$. We now state the classical Lang-Weil
inequality, except that we give an explicit bound for the remainder.

\begin{theorem}
\label{LangWeil}
Let $X$ be a projective algebraic subvariety in $\ProjCan_{k}$, of dimension
$n$ and of degree $d$, defined over $k$. Then
$$
\abs{\card{X(k)} - \pi_{n}} \leq (d - 1)(d - 2) q^{n - (1/2)} +
C_{+}(X) \, q^{n - 1},
$$
where $C_{+}(X)$ depends only on $\Bar{X}$, and
$$
C_{+}(X) \leq 9 \times 2^{m} \times (m \delta + 3)^{N + 1},
$$
if $X$ is of type $(m, N, \mathbf{d})$, with $\mathbf{d} = (d_{1}, \dots,
d_{m})$ and $\delta = \max(d_{1}, \dots, d_{m})$. In particular,
$C_{+}(X)$ is bounded by a quantity which is independent of the field
$k$.
\end{theorem}

In order to prove this, we need a preliminary result.

\begin{proposition}
\label{RoughLangWeil}
Let $X$ be a projective algebraic subvariety in $\ProjCan_{k}$ of
dimension $n$ defined over $k$. Then
$$
\abs{\card{X(k)} - \pi_{n} - \Tr(F \mid H^{2n - 1}_{+}(\bar{X},
\QQ_{\ell})) }
\leq C_{+}(X) \, q^{n - 1},
$$
where $C_{+}(X)$ is as in Theorem \ref{LangWeil}.
\end{proposition}

\begin{proof}
Denote by $H^{2n - 1}_{-}(\bar{X})$ the subspace of $H^{2n -1}(\bar{X})$
corresponding to eigenvalues of the Frobenius endomorphism of weight
strictly smaller than $2n - 1$ and define
$$
C_{+}(X) = \dim H^{2n - 1}_{-}(\bar{X}) + \
\sum_{i = 0}^{2n - 2} b_{i,l}(\bar{X}) + \varepsilon_{i}.
$$
Clearly, $C_{+}(X)$ depends only on $\Bar{X}$, and by Proposition
\ref{SchKatz},
$$
C_{+}(X) \leq \tau(X) \leq \tau_{k}(m, N, \mathbf{d}) \leq
9 \times 2^{m} \times (m \delta + 3)^{N + 1},
$$
with the notations introduced therein. Then the Trace Formula
\eqref{GrTF} and Deligne's Main Theorem \eqref{DeMT} imply the required
result.
\end{proof}

\begin{proof}[Proof of Theorem \ref{LangWeil}]
By Theorem \ref{GenConj}\eqref{CorrConj1} and Proposition
\ref{ChowSerreWeil}\eqref{CSW3},
$$
b_{2n - 1}^{+}(X) =
2 \dim \Alb_{w} X \leq g(\widetilde{Y}) \leq (d - 1)(d - 2),
$$
if $Y$ is a suitably chosen linear section of $X$ of dimension $1$, and if
$\widetilde{Y}$ is a nonsingular curve birationally equivalent to $Y$. Since
all the eigenvalues of the Frobenius automorphism in $H^{2n -
1}_{+}(\bar{X}, \QQ_{\ell})$ are pure of weight $2n - 1$ by definition, we
obtain by Proposition \ref{RoughLangWeil}:
$$
\abs{\card{X(k)} - \pi_{n}} \leq 2g(\widetilde{Y}) q^{n - (1/2)} +
C_{+}(X) \, q^{n - 1},
$$
which is better than the desired inequality.
\end{proof}

\begin{remark}
\label{AffineLangWeil}
In exactly the same way, applying \cite[Th. 1]{Katz4}, one establishes that
if $X$ is a closed algebraic subvariety in $\Aff^{N}_{k}$, of dimension
$n$, of degree $d$ and of type $(m, N, \mathbf{d})$, defined over $k$, then
$$
\abs{\card{X(k)} - q^{n}} \leq (d - 1)(d - 2) q^{n - (1/2)} +
C_{+}(X) \, q^{n - 1},
$$
where $C_{+}(X)$ depends only on $\bar{X}$ and is not greater than
$6 \times 2^{m} \times (m \delta + 3)^{N + 1}$.
\end{remark}

\begin{remark}
As a consequence of Remark \ref{AffineLangWeil}, we easily obtain the
following version of a lower bound due to W. Schmidt \cite{Schmidt} for the
number of points of affine hypersurfaces. If $f \in k[T_{1}, \dots , T_{N}]$
is an absolutely irreducible polynomial of degree $d$, and if $X$ is the
hypersurface in $\Aff^{N}_{k}$ with equation $f(T_{1}, \dots , T_{N}) = 0$,
then
$$
\card{X(k)} \geq
q^{N - 1} - (d - 1)(d - 2) q^{N - (3/2)} - 12 (d + 3)^{N + 1} q^{N - 2}.
$$
It may be noted that with the help of Schmidt's bound, one is able to
replace $12 (d + 3)^{N + 1}$ by a much better constant, namely $6d^{2}$, but
his bound is only valid for large values of $q$.
\end{remark}

Proposition \ref{RoughLangWeil} gives the second term in the asymptotic
expansion of $\card{X(\FF_{q^{s}})}$ when $s$ is large. From Theorem
\ref{GenConj}, we deduce immediately the following precise inequality, which
involves only purely algebraic terms.

\begin{corollary}
\label{TightLangWeil}
Let $X$ be a projective variety of dimension $n \geq 2$ defined over $k$,
and let $\varphi$ be the Frobenius endomorphism of $\Alb_{w} X$. Then
$$
\abs{\card{X(k)} - \pi_{n} + q^{n - 1} \Tr(\varphi) }
\leq
C_{+}(X) \, q^{n - 1},
$$
where $C_{+}(X)$ is as in Theorem \ref{LangWeil}. \hfill \qed
\end{corollary}

\begin{corollary}
\label{CritPic}
With hypotheses as above, the following are equivalent:
\begin{enumerate}
\item
\label{CritPic1}
There is a constant $C$ such that, for every $s \geq 1$,
$$
\abs{\card{X(\FF_{q^{s}})} - q^{ns}}
\leq C q^{s(n - 1)}.
$$
\item
\label{CritPic2}
The Albanese-Weil variety of $X$ is trivial.
\end{enumerate}
\end{corollary}

\begin{proof}
By Lemma \ref{CritCoh} we know that \eqref{CritPic1} holds if and only if
$H^{2n - 1}_{+}(\bar{X}) = 0$, and this last condition is equivalent to
\eqref{CritPic2} by Theorem \ref{GenConj}.
\end{proof}

For instance, rational varieties, and, by Remark \ref{CaseR2} or
\ref{TrivAlb}, complete intersections regular in codimension $2$
satisfy the conditions of the preceding Lemma.

It remains to prove Theorem \ref{GenConj}. For that purpose, we need the
following non-effective estimate \cite[p. 333]{BS}.

\begin{lemma}[Bombieri and Sperber]
\label{BomSper}
Let $X$ be a projective variety of dimension $n \geq 2$ defined over $k$,
and let $\varphi$ as above. Then
$$
\card{X(k)} = \pi_{n} - q^{n - 1} \Tr(\varphi) + O(q^{n - 1}).
$$
\end{lemma}

\begin{proof}
By induction on $n = \dim X$. If $\dim X = 2$, then $\resolTwo$ is
satisfied. Hence, the conclusions of Theorem \ref{GenConj} hold true by
Proposition \ref{CorrConj}, and
$$
\abs{\card{X(k)} - \pi_{2} + q \Tr(\varphi) } \leq C_{+}(X) \, q^{},
$$
by Proposition \ref{RoughLangWeil} and Theorem \ref{GenConj}\eqref{CorrConj2}.
Suppose that $n \geq 3$ and that the lemma is true for any projective variety
of dimension $n - 1$. We can assume that $X$ is a projective algebraic
subvariety in $\ProjCan_{k}$ not contained in any hyperplane. In view of Lemma
\ref{Bertini}, if $q$ is sufficiently large, there is a linear subvariety $E$
of $\ProjCan_{k}$, of codimension $2$, defined over $k$, such that $\dim X \cap
E = n - 2$. For any $u \in \PP^{1}_{k}$, the reciprocal image of $u$ by the
projection
$$\pi : \ProjCan_{k} - E \longrightarrow \PP^{1}_{k}$$ is a hyperplane
$H_{u}$ containing $E$, and $Y_{u} = X \cap H_{u}$ is a projective algebraic
set of dimension $\leq n - 1$ and of degree $\leq
\deg X$. Now we have
$$
X(k) = \bigcup_{u \in \, \PP^{1}(k)} Y_{u}(k),
$$
hence
$$
\card{X(k)} = \sum_{u \in \, \PP^{1}(k)} \, \card{Y_{u}(k)}
+ O(q^{n - 1}) \, ;
$$
in fact the error term is bounded by $(q + 1) \card{X(k) \cap E(k)}$,
and can be estimated by Proposition \ref{BoundAlgSet} below. Let $S$ be
the set of $u \in \PP^{1}_{k}$ such that $Y_{u}$ is not a subvariety of
dimension $n - 1$, or such that the canonical morphism
$$\lambda : \Alb_{w} Y_{u} \longrightarrow \Alb_{w} X$$
is not a purely inseparable isogeny. By Lemma \ref{Bertini} and
Proposition \ref{ChowSerreWeil}\eqref{CSW1}, the set $S$ is finite and
applying Proposition \ref{BoundAlgSet} again, we get
$$\sum_{u \in S} \, \card{Y_{u}(k)} = O(q^{n - 1}),$$
hence,
$$
\card{X(k)} = \sum_{u \notin S} \, \card{Y_{u}(k)} + O(q^{n - 1}).
$$
Let $\varphi_{u}$ be the Frobenius endomorphism of $\Alb_{w} Y_{u}$. By the
induction hypothesis,
$$
\card{Y_{u}(k)} = \pi_{n - 1} - q^{n - 2} \Tr(\varphi_{u}) + O(q^{n -
2}).
$$
But if $u \notin S$ then $\Tr(\varphi_{u}) = \Tr(\varphi)$ by hypothesis.
The two preceding relations then imply
$$
\card{X(k)} = (q + 1 - \card{S})(\pi_{n - 1} - q^{n - 2} \Tr(\varphi))
+ O(q^{n - 1}),
$$
and the lemma is proved.
\end{proof}

\begin{proof}[Proof of Theorem \ref{GenConj}]
From Proposition \ref{RoughLangWeil} and from Lemma \ref{BomSper} we
get
$$
\Tr(F \mid H^{2n - 1}_{+}(\bar{X}, \QQ_{\ell})) - q^{n - 1} \Tr(\varphi)
= O(q^{n - 1})
$$
But the eigenvalues of $F$ in $H^{2n - 1}_{+}(\bar{X}, \QQ_{\ell})$ and the
numbers $q^{n - 1} \alpha_{j}$, where $(\alpha_{j})$ is the family of
characteristic roots of $\Alb_{w} X$, are pure of weight $2n -1$. From Lemma
\ref{Cauchy} we deduce that these two families are identical, and
$$
\Tr(F \mid H^{2n - 1}_{+}(\bar{X}, \QQ_{\ell})) =
q^{n - 1} \Tr(\varphi).
$$
This yields the desired result.
\end{proof}

One can also improve the Lang-Weil inequality when the varieties are of
small codimension:

\begin{corollary}
\label{SmallCodimBound}
Let $X$ be a projective subvariety in $\ProjCan_{k}$, and assume
$$\dim \sing X \leq s, \qquad \codim X \leq \dim X - s - 1.$$
If $\dim X = n$, then
$$
\abs{\card{X(k)} - \, \pi_{n}} \leq C_{N + s}(\bar{X}) \, q^{(N + s)/2},
$$
where $C_{N + s}(\bar{X})$ is as in Theorem \ref{MainThm}.
\end{corollary}

\begin{proof}
From Barth's Theorem \cite[Thm. 6.1, p. 146]{Ha2} and Theorem
\ref{LefschetzHigh}, we deduce that \eqref{HomZero} holds if $i \geq N + s +
1$. Now apply the Trace Formula \eqref{GrTF} and Deligne's Main Theorem
\eqref{DeMT}.
\end{proof}

\section{Number of Points of Algebraic Sets}
\label{AlgSets}

In most applications, it is useful to have at one's disposal some bounds on
general algebraic sets. If $X \subset \ProjCan$ is a projective algebraic
set defined over a field $k$ we define the \emph{dimension} (resp. the
\emph{degree}) of $X$ as the maximum (resp. the sum) of the dimensions (resp.
of the degrees) of the $k$-irreducible components of $X$. From now on, let
$k = \FF_{q}$. The following statement is a quantitative version of Lemma $1$
of Lang-Weil \cite{LangWeil} and generalizes Proposition 2.3 of \cite{Lachaud}.

\begin{proposition}
\label{BoundAlgSet}
If $X \subset \ProjCan$ is a projective algebraic set defined over $k$ of
dimension $n$ and of degree $d$, then
$$\card{X(k)} \leq d \pi_{n}.$$
\end{proposition}

\begin{proof}
By induction on $n$. Recall that an algebraic set $X \subset \ProjCan$ is
\emph{nondegenerate} in $\ProjCan$ if $X$ is not included in any hyperplane,
\idest, if the linear subvariety generated by $X$ is equal to $\ProjCan$. If
$n = 0$ then $\card{X(k)} \leq d$. Assume now that $n \geq 1$. We first
prove the desired inequality when $X$ is $k$-irreducible. Let $E$ be the
linear subvariety in $\ProjCan$, defined over $k$, generated by $X$; set $m
= \dim E \geq n$ and identify $E$ with $\PP^{m}$. Then $X$ is a
nondegenerate $k$-irreducible subset in $\PP^{m}$. Let
$$
T = \set{(x, H) \in X(k) \times (\PP^{m})^{*}(k)}{x \in X(k) \cap H(k)}.
$$
We get a diagram made up of the two projections
$$
\begin{array}{lcl}
            & T &                      \\
^{p_{1}} \swarrow &   & \searrow^{p_{2}}     \\
X(k)              &   & (\PP^{m})^{*}(k)
\end{array}
$$
If $x \in X(k)$ then $p_{1}^{-1}(x)$ is in bijection with the set of
hyperplanes $H \in (\PP^{m})^{*}(k)$ with $x \in H(k)$; hence
$\card{p_{1}^{-1}(x)} = \pi_{m - 1}$ and
\begin{equation}
\label{T1}
\card{T(k)} = \pi_{m - 1} \card{X(k)}.
\end{equation}
On the other hand, if $H \in (\PP^{m})^{*}(k)$, then $p_{2}^{-1}(H)$ is
in bijection with to $X(k) \cap H(k)$, hence
\begin{equation}
\label{T2}
\card{T(k)} = \sum_{H} \card{X(k) \cap H(k)},
\end{equation}
where $H$ runs over the whole of $(\PP^{m})^{*}(k)$. Since $X$ is a
nondegenerate $k$-irreducible subset in $\PP^{m}$, every $H \in
(\PP^{m})^{*}(k)$ properly intersects $X$ and the hyperplane section $X \cap
H$ is of dimension $\leq n - 1$. Moreover, such a hyperplane section is of
degree $d$ by B\'ezout's Theorem. Thus by the induction hypothesis,
$$\card{X(k) \cap H(k)} \leq d \pi_{n - 1},$$
and from \eqref{T1} and \eqref{T2} we deduce
$$\card{X(k)} \leq \frac{\pi_{m}}{\pi_{m - 1}} \, d \pi_{n - 1}.$$
Now it is easy to check that if $m \geq n$, then
$$
q \leq \frac{\pi_{m}}{\pi_{m - 1}} \leq
\frac{\pi_{n}}{\pi_{n - 1}} \leq q + 1,
$$
so the desired inequality is proved when $X$ is irreducible of dimension
$n$. In the general case, let
$$X = Y_{1} \cup \dots \cup Y_{s}$$
be the irredundant decomposition of $X$ in $k$-irreducible components, in
such a way that
$$
\dim Y_{i} \leq n, \quad d_{1} + \dots + d_{s} = d, \quad
(\deg Y_{i} = d_{i}).
$$
Then
$$
\card{X(k)} \leq \card{Y_{1}(k)} + \dots + \card{Y_{s}(k)} \leq
(d_{1} + \dots + d_{s}) \pi_{n} = d \pi_{n}. \quad \qed
$$
\renewcommand{\qed}{}
\end{proof}
\renewcommand{\qed}{\qedsymbol}

We take this opportunity to report the following conjecture on the
number of points of complete intersections of small codimension.

\begin{conjecture}[Lachaud]
\label{LachaudConj}
If $X \subset \ProjCan_{k}$ is a projective algebraic set defined
over $k$ of dimension $n \geq N/2$ and of degree $d \leq q + 1$ which is
a complete intersection, then
$$
\card{X(k)} \leq d \pi_{n} - (d - 1)\pi_{2n - N}
= d(\pi_{n} - \pi_{2n - N}) + \pi_{2n - N}.
$$
\end{conjecture}

\begin{remark}
\label{LachaudConjRem}
The preceding conjecture is true in the following cases:
\begin{enumerate}
\item
\label{Serre}
$X$ is of codimension $1$.
\item
\label{Linear}
$X$ is a union of linear varieties of the same dimension.
\end{enumerate}
Assertion \eqref{Serre} is \emph{Serre's inequality} \cite{Serre3}: if
$X$ is an hypersurface of dimension $n$ and of degree $d \leq q + 1$,
then
$$\card{X(k)} \leq d q^{n} + \pi_{n - 1}.$$
Now assume that $X$ is the union of $d$ linear varieties $G_{1}, \dots,
G_{d}$ of dimension $n \geq N/2$. We prove \eqref{Linear} by induction on
$d$. Write $G_{i}(k) = G_{i} \ (1 \leq i \leq d)$ for brevity. If $d = 1$
then
$$\card{G_{1}} = \pi_{n} = (\pi_{n} - \pi_{2N - n}) + \pi_{2N - n},$$
and the assertion is true. Now if $G_{1}$ and $G_{2}$ are two linear
varieties of dimension $n$, then $\dim G_{1} \cap G_{2} \geq 2n - N$. Hence
for $d > 1$,
$$
\card{G_{d} \cap (G_{1} \cup \dots \cup G_{d - 1})} \geq \pi_{2n - N}.
$$
Now note that
$$
\card{G_{1} \cup \dots \cup G_{d}} = \card{G_{1} \cup \dots \cup G_{d - 1}}
+ \card{G_{d}} - \card{G_{d} \cap (G_{1} \cup \dots \cup G_{d - 1})}.
$$
If we apply the induction hypothesis we get
\begin{eqnarray*}
\card{G_{1} \cup \dots \cup G_{d}} & \leq &
(d - 1)(\pi_{n} - \pi_{2n - N}) + \pi_{2n - N} + \pi_{n} - \pi_{2n - N}\\
& = & d(\pi_{n} - \pi_{2n - N}) + \pi_{2n - N},
\end{eqnarray*}
which proves the desired inequality.
\end{remark}

\section*{Acknowledgements}

This research was partly supported by the Indo-French Mathematical Research
Program of the Centre National de la Recherche Scientifique (CNRS) of France
and the National Board for Higher Mathematics (NBHM) of India, and we thank
these organisations for their support. We would also like to express our
warm gratitude to Jean-Pierre Serre for his interest, and to Alexe\"{\i}
Skorobogatov for his valuable suggestions. Thanks are also due to the
referees for pointing out some corrections in earlier versions of this paper
and making suggestions for improvements.

\renewcommand{\refname}{Abbreviations}

\end{document}